\documentclass[12pt,reqno]{amsart}

\usepackage{amssymb}

\usepackage{hyperref}

\usepackage{epsfig}
\usepackage{epic, eepic} 
\usepackage{tikz, tkz-tab, tikz-cd, pgfplots}
\usetikzlibrary{intersections}

\usepackage{caption}

\newtheorem{theorem}{Theorem}[section]
\newtheorem{prop}[theorem]{Proposition}
\newtheorem{lemma}[theorem]{Lemma}
\newtheorem{cor}[theorem]{Corollary}
\newtheorem{question}[theorem]{Question}
\newtheorem{obs}[theorem]{Observation}

\newtheorem{mainthm}{Theorem}

\theoremstyle{definition}
\newtheorem{definition}[theorem]{Definition}
\newtheorem{remark}[theorem]{Remark}
\newtheorem{example}[theorem]{Example}

\numberwithin{equation}{section}

\newcommand{\A}{\mathbb{A}}
\newcommand{\C}{\mathbb{C}}

\newcommand{\N}{\mathbb{N}}
\newcommand{\Q}{\mathbb{Q}}
\newcommand{\R}{\mathbb{R}}

\newcommand{\Z}{\mathbb{Z}}
\newcommand{\Kbar}{\overline{\kappa}}

\newcommand{\PP}{\mathbb{P}}

\newcommand{\wt}{\widetilde}
\newcommand{\ol}{\overline}

\newcommand{\Hom}{\operatorname{Hom}}

\newcommand{\rank}{\operatorname{rank}}
\newcommand{\mult}{\operatorname{mult}}
\newcommand{\length}{\operatorname{length}}

\newcommand{\id}{\operatorname{id}}
\newcommand{\Ima}{\operatorname{Im}}

\usetikzlibrary{decorations.markings}
\tikzstyle{vertex}=[circle, draw, inner sep=0pt, minimum size=5pt]

\usepackage{times}

\newcommand{\FujitahGrid}[1]{\begin{tikzpicture}[sibling distance=1cm, level distance=1cm, baseline, x=0.7cm, y=1.5cm]
	\draw[line width=0.3mm, black]
	(-1.5,1) -- (2.5,1) node[right] {$M_n$}
	(-1.5,-1) -- (2.5,-1) node[right] {$\ol{M}_n$}
	(-1,1.2) node[above] {$\ell_0$} -- (-1,-1.2)
	(0,1.2) node[above] {$\ell_1$} -- (0,-1.2)
	(1,1.2) node[above] {$\ell_2$} -- (1,-1.2);
	\filldraw[black] (0,-1) circle(2pt) node[above right] {$p_1$};
	\filldraw[black] (1,-1) circle(2pt) node[above right] {$p_2$};
	\end{tikzpicture}}

\newcommand{\Fujitamiddle}[1]{\begin{tikzpicture}[sibling distance=1cm, level distance=1cm, baseline, x=0.8cm, y=0.7cm]
	\draw[line width=0.3mm, black]
	(-3,2.5) -- (3,2.5) node[right] {$M_n$} node[pos=0.9,above]{$-n$}
	
	(-3,-2.5) -- (3,-2.5) node[right] {$\overline{M}_n$} node[pos=0.9,below]{$n-2$}
	
	(-2.5,3) node[above] {$\ell_0$} -- (-2.5,-3) node[pos=0.5,left]{$0$};
	
	\draw[line width=0.3mm, black, name path=x]
	(-0.5,3) .. controls (-1,1.5) .. (-2,0) node[pos=0,left]{$\ell_1'$} node[pos=0.5,right]{$-1$};
	
	\draw[line width=0.3mm, black, name path=y]
	(1.5,3) .. controls (1,1.5) .. (0,0) node[pos=0,left]{$\ell_2'$}
	node[pos=0.5,right]{$-1$};
	
	\draw[line width=0.3mm, black, dashed, name path=x'] 
	(-2,0.5) .. controls (-1,-1.5) .. (-0.5,-3) node[pos=1,left]{$D_1$}node[pos=0.5,right]{$-1$};
	
	\draw[line width=0.3mm, black, dashed, name path=y']
	(0,0.5) .. controls (1,-1.5) .. (1.5,-3) node[pos=1,left]{$D_2$}node[pos=0.5,right]{$-1$};
	
	\fill [black, name intersections={of=x and x'}] (intersection-1) circle (2pt) node[right] {$\;q_1$};
	
	\fill [black, name intersections={of=y and y'}] (intersection-1) circle (2pt) node[right] {$\;q_2$};
	\end{tikzpicture}}

\newcommand{\Fujitah}[1]{\begin{tikzpicture}[sibling distance=1cm, level distance=1cm, baseline, x=0.8cm, y=0.7cm]
	\draw[line width=0.3mm, black]
	(-3,2.5) -- (3,2.5) node[right] {$M_n$} node[pos=0.9,above]{$-n$}
	
	(-3,-2.5) -- (3,-2.5) node[right] {$\overline{M}_n$} node[pos=0.9,below]{$n-2$}
	
	(-2.5,3) node[above] {$\ell_0$} -- (-2.5,-3) node[pos=0.5,left]{$0$}
	
	(-0.5,3) .. controls (-1,1.5) .. (-2,0.5) node[pos=0,left]{$\ell_1''$}
	node[pos=0.5,right]{$-2$}
	
	(1.5,3) .. controls (1,1.5) .. (0,0.5) node[pos=0,left]{$\ell_2''$}
	node[pos=0.5,right]{$-2$}
	
	(-2,-0.5) .. controls (-1,-1.5) .. (-0.5,-3) node[pos=1,left]{$D_1'$}
	node[pos=0.5,right]{$-2$}
	
	(0,-0.5) .. controls (1,-1.5) .. (1.5,-3) node[pos=1,left]{$D_2'$}
	node[pos=0.5,right]{$-2$};

	\draw[line width=0.3mm, black, dashed] 
	(-2,1) .. controls (-1.4,0) .. (-2,-1) node[pos=0.5,left]{$E_1$}
	node[pos=0.7,right]{$-1$}
	
	(0,1) .. controls (0.6,0) .. (0,-1) node[pos=0.5,left]{$E_2$}
	node[pos=0.7,right]{$-1$};
	\end{tikzpicture}}

\begin{document}

\title[On Stein spaces with finite homotopy rank-sum]{On Stein spaces with finite homotopy rank-sum}

\author[I. Biswas]{Indranil Biswas}

\address{Department of Mathematics, Shiv Nadar University, NH91, Tehsil
Dadri, Greater Noida, Uttar Pradesh 201314, India}

\email{indranil.biswas@snu.edu.in, indranil29@gmail.com}

\author[B. Hajra]{Buddhadev Hajra}

\address{School of Mathematics, Tata Institute of Fundamental
Research, Homi Bhabha Road, Mumbai 400005, India}

\email{hajrabuddhadev92@gmail.com}



\begin{abstract}
A topological space (not necessarily simply connected) is said to have \emph{finite homotopy rank-sum} if the sum of the ranks of all higher
homotopy groups (from the second homotopy group onward) is finite. In this article, we consider Stein spaces of arbitrary dimension satisfying
the above rational homotopy theoretic property, although most of this article focuses on Stein surfaces only. We characterize
all Stein surfaces satisfying the finite homotopy rank-sum property. In particular, if such a
Stein surface is affine and every element of its fundamental group is finite, it is either simply connected
or has a fundamental group of order $2$. A detailed classification of the smooth complex
affine surfaces of the non-general type satisfying the finite homotopy rank-sum property is obtained. It turns out
that these affine surfaces are Eilenberg--MacLane spaces whenever the fundamental group is infinite.
\end{abstract}

\maketitle

\tableofcontents

\section{Introduction}\label{se1}

The higher homotopy groups of any topological space $X$ are always abelian, and $\rank(\pi_i(X))$ is defined to be 
the dimension of the $\Q$-vector space $\pi_i(X)\otimes_\Z\Q$. The rank of the second homotopy group of the wedge sum $S^1\vee S^2$ is infinite.
However, in \cite{Ser}, J.-P. Serre proved that all the higher homotopy groups are finitely generated for a simply connected finite CW complex
$X$ (see also \cite[p.~504]{Spa}). More generally, if $\pi_1(X)$ is a
finite group, all the higher homotopy groups of $X$ are finitely generated. 


A simply connected topological space $X$ is of \emph{rationally elliptic homotopy type} (cf. \cite[Part VI, \S~32]{Fel-Hal-Tho2}) if it satisfies the
following two conditions:
\begin{align}
 &\sum\limits_{i\,\geq\,2}\rank(\pi_i(X))\,=\,\sum\limits_{i\,\geq\,2}\dim \pi_i(X)\otimes_\Z \Q\,<\, \infty
,\label{Defn: Finite homotopy rank-sum}
\end{align}
and
\begin{align}
&\sum\limits_{i\,\geq\,0}b_i(X)\,=\,\sum\limits_{i\,\geq\, 0}\dim(H^i(X;\,\Q))\,<\,\infty,
\end{align}
where $b_i(X)$ denotes the $i$-th Betti number of $X$.

A topological space $X$ (not necessarily simply connected) is said to have \emph{finite homotopy rank-sum} or is said to \emph{satisfy the finite 
homotopy rank-sum property} if $$\sum\limits_{i\,\geq\,2}\rank(\pi_i(X))\,=\,\sum_{i\,\geq\,2}\dim \pi_i(X)\otimes_\Z \Q \,<\, \infty$$ (see \eqref{Defn: Finite homotopy rank-sum}) or equivalently if
\begin{enumerate}
\item[(a)] $\dim \pi_i(S)\otimes_\Z \Q \,<\, \infty$ for all $i\,\geq\, 2$, and
\item[(b)] there is an integer $N\,\geq\, 2$ such that $\dim \pi_i(S)\otimes_\Z \Q \,=\, 0$ for all $i\,\geq\, N$, i.e., all but finitely
many higher homotopy groups are torsion abelian groups.
\end{enumerate}

An example of elliptic homotopy type topological space is the real sphere $S^n$ of any positive dimension $n$. Any
complex rational homogeneous space is of elliptic homotopy type (see \cite[Example 2.5]{Amo-Bis}). Compact K{\"a}hler manifolds of the elliptic homotopy type of complex dimension up to $3$ were described in \cite{Amo-Bis}.

Our aim here is to characterize all complex Stein surfaces (not necessarily normal) having
finite homotopy rank-sum. For the general definition of a Stein complex analytic space, we refer to \cite[Chapter IX, \S~4]{Sha}.
In particular, we will classify all smooth complex affine surfaces of non-general type whose universal cover is of elliptic homotopy type.

The following dichotomy result for complex Stein surfaces extends a standard dichotomy result
for simply connected ($1$-connected) finite CW complexes (the result for simply connected finite CW complexes
is recalled in Theorem \ref{Thm: E-H dichotomy for finite complexes}).

\begin{mainthm}[{Theorem \ref{Thm: Marked dichotomy for Stein surfaces}}]\label{thm:mainA}
Let $X$ be a complex connected Stein surface. Then the higher homotopy groups of $X$ satisfy one of the two mutually exclusive properties:
\begin{enumerate}
\item[(i)] $\sum_{i\geq2}\dim \pi_i(X)\otimes_\Z \Q \,\leq\, 2$.

\item[(ii)] $\sum_{i=2}^{k}\dim \pi_i(X)\otimes_\Z \Q$ grows exponentially in $k$, i.e., there exists a real number $C\,>\,1$ such that $\sum_{i=2}^{k}
\dim \pi_i(X)\otimes_\Z \Q \,>\, C^k$ for all integers $k$ large enough.
\end{enumerate}
\end{mainthm}

Section \ref{se4} considers complex connected Stein spaces of arbitrary dimension. Under the assumption that the total rank of the higher homotopy groups is finite, a bound for the total rank has been found. The following are the two main results in this direction.

\begin{mainthm}[{Theorem \ref{Thm: Bound of the rank of homotopy graded algebra for Stein spaces of dimension up to 3}}]\label{thm:mainB}
Let $X$ be a complex connected Stein space of dimension at most $3$. Then $X$ satisfies the finite homotopy rank-sum property
if and only if $$\sum_{i\geq2}\dim \pi_i(X)\otimes_\Z \Q \,\leq\, \dim X.$$
\end{mainthm}

\begin{mainthm}[{Theorem \ref{Thm: Bound of the rank of homotopy graded algebra for certain Stein spaces of dimension at least 4}}]\label{thm:mainC}
Let $X$ be a complex connected Stein space of dimension at least $4$. Then the following assertions hold. 

\begin{enumerate}
\item Assume that 
\begin{equation*}
\pi_i(X)\otimes_\Z \Q \,=\,0, \quad \forall\ \ \, 2\,\leq\,i\,\leq\, \bigg\lfloor \frac{\dim X}{2} \bigg\rfloor,
\end{equation*}
where $\lfloor\,.\,\rfloor$ denotes the greatest integer function. Then $X$ satisfies the finite homotopy rank-sum property if and only if 
$$\sum_{i\geq2}\dim \pi_i(X)\otimes_\Z \Q \,\leq\, \dim X.$$
\item Moreover, if $\dim X\,=\,4$ and $\pi_2(X)\otimes_\Z\Q\,=\,0$, then $X$ satisfies the finite homotopy rank-sum property
if and only if $$\sum_{i\geq2}\dim \pi_i(X)\otimes_\Z \Q \,\leq\, 3.$$
\end{enumerate}
\end{mainthm}

The following two are the main results in Section \ref{se5}.

\begin{mainthm}[{Corollary \ref{Cor: Characterization of E-type Stein surfaces}}]\label{thm:mainD}
Let $X$ be a complex connected Stein surface. Then $X$ satisfies the finite homotopy rank-sum property if and only if one of the following two statements hold:
\begin{enumerate}
\item[(i)] $\pi_2(X)\,=\,0$, {\underline{\rm or}} equivalently, $X$ is an Eilenberg--MacLane $K(\pi_1(X),\,1)$-space.

\item[(ii)] $\pi_2(X)\,=\,\Z$, {\underline{\rm or}} equivalently the universal covering of $X$ is homotopic to $S^2$.
\end{enumerate}
\end{mainthm}

\begin{mainthm}[{Corollary \ref{Cor:Finite Fundamental Group of Elliptic Type Affine Surface}}]\label{thm:mainE}
Let $X$ be a complex affine surface satisfying the finite homotopy rank-sum property. If $X$ is not simply connected, and
every element of $\pi_1(X)$ has finite order, then $\pi_1(X)\,\cong\,\Z/2\Z$ and the Euler-Poincar\'e characteristic of $X$ is $1$. 
\end{mainthm}

Section \ref{se6} is entirely devoted to the smooth complex affine surfaces of the non-general type for
which the sum of the ranks of the higher homotopy groups is finite. The following are the two main results proved there.

\begin{mainthm}[{Proposition \ref{Prop: Infinite pi_1 and finite rank pi_2 imply EM-ness for smooth affine non-general type surfaces}}]\label{thm:mainF}
Let $X$ be a smooth complex affine surface of the non-general type having an infinite fundamental group. If $\pi_2(X)$ is
of finite rank, then $X$ is an Eilenberg--MacLane $K(\pi,\,1)$-space.
\end{mainthm}

\begin{mainthm}[{Theorem \ref{Classification of E-type affine surfaces of non-general type}}]\label{thm:mainG}
Let $X$ be a smooth complex affine surface of non-general type satisfying the finite homotopy rank-sum property. Then
one of the following three statements holds:
\begin{enumerate}
\item[(i)] $X$ is an Eilenberg--MacLane $K(\pi_1(X),\,1)$-space.

\item[(ii)] $X$ is a smooth affine $\Q$-homology plane with $\pi_1(X)\,=\,\Z/2\Z$.

\item[(iii)] $X$ is homotopically equivalent to the sphere $S^2$.
\end{enumerate}
\end{mainthm}

A $\Z$-homology (respectively, $\Q$-homology) plane is defined to be a smooth irreducible complex affine algebraic surface $V$ such that 
$H_i(V;\,\Z)\,=\,0$ (respectively, $H_i(V;\,\Q)\,=\,0$) for all $i \,\geq\,1$. The class of smooth affine ($\Z$ or $\Q$)--homology 
planes have been of special interest. The condition that the sum of the ranks of the higher homotopy groups is finite seems to be very 
restrictive for them. In Section \ref{se7}, the following result is proved in this direction.

\begin{mainthm}[{Corollary \ref{Cor: Smooth Q-homology plane of E-type}, 
Corollary \ref{Cor: No non-contractible smooth Z-homology plane of E-type}}]\label{thm:mainH}
Let $X$ be a smooth $\Q$-homology plane of non-general type satisfying the finite homotopy rank-sum property. Then $X$ is either contractible or a smooth
affine $\Q$-homology plane with $\pi_1(X)\,=\,\Z/2\Z$.
 
Moreover, if $X$ is a $\Z$-homology plane, then it is contractible.
\end{mainthm}
 
\section{Preliminaries}\label{se2}
Firstly, we will list certain notations and conventions that we follow throughout this article.

\begin{enumerate}
\item[$\bullet$] All the homology and cohomology groups are with integer coefficients unless the coefficient group is specified.

\item[$\bullet$] $H^i_c(X;\,A)$ (for $i\,\in \,\Z_{\geq 0}$) : $i$-th compactly supported cohomology group of a
topological space $X$ with coefficients in an abelian group $A$.

\item[$\bullet$] For a given subspace $Z$ of $X$, the standard notations for the homotopy and homology groups of the pair $(X,\,Z)$
are used.

\item[$\bullet$] If $X$ is a simplicial complex or a finite CW complex, then $\wt{H}_i(X;\,\Z)$ (for $i\,\in \,\Z_{\geq 0}$) denotes 
the $i$-th reduced integral homology group of $X$.

\item[$\bullet$] By ``higher homotopy groups'' of $X$ we always mean all $\pi_i(X)$,\, $i\,\geq\, 2$.

\item[$\bullet$] Whenever we consider smooth algebraic surfaces and their morphisms, these are defined over $\C$.

\item[$\bullet$] The $n$-dimensional affine space and projective space defined over $\C$ are denoted by $\A^n$ and $\PP^n$
respectively.

\item[$\bullet$] For a smooth algebraic surface $S$ we employ the following notation.\\
\indent $\Kbar(S)$ : The logarithmic Kodaira dimension of $S$.\\
For the definition of $\Kbar$, see \cite{Iit}.

\item[$\bullet$] For a (finite) CW-complex $X$ we employ the following notations:\\
\indent $b_i(X) \; (\text{for } i \,\in \,\Z_{>0})$ : The $i$-th Betti number $\dim_{\C}H_i(X, \C)$.\\
\indent $e(X)$ : The Euler--Poincar{\'e} characteristic of $X$, which is defined as
$$e(X)\ :=\ \sum\limits_{i= 0}^{\infty}(-1)^i b_i(X).$$
\end{enumerate}

\subsection{Definitions}

We will recall the definition of certain topological spaces which will be appearing in the subsequent sections.

\begin{definition}[{\bf Eilenberg--MacLane spaces $K(G,\,n)$}]\label{Def: EM Spaces}\mbox{}
A path-connected topological space $X$ is said to be an Eilenberg--MacLane space if for some integer $n \in \Z_{>0}$ and a group $G$,
$$\pi_k(X)\,=\,
\begin{cases}
G, &\text{ for }\, k\,=\,n,\\
0, &\text{ otherwise},
\end{cases}
$$
and it is denoted by $K(G,\, n)$.

For any abelian group $G$ and any integer $n\,\in\, \N$, in the category of CW-complexes an Eilenberg--MacLane $K(G,\,n)$ space exists
and it is unique up to homotopy (cf. \cite[Chapter 8, \S1, Corollary 5]{Spa}, cf. \cite[Corollary 2]{Mos-Tan}). So we
use the same notation $K(G,\,n)$ to denote any CW complex belonging to this homotopy class.
\end{definition}

\begin{definition}[{\bf Moore spaces $M(G,n)$}]\label{Def: Moore Spaces}
A CW-complex $X$ is said to be an $M(G,\,1)$-space for an abelian group $G$ if

$$\widetilde{H}_i(X;\,\Z)=
	\begin{cases}
	G, &\text{for }\, i=1,\\
	0, &\text{otherwise}.
	\end{cases}
	$$
Also, for an integer $n \,\geq\, 2$ and an abelian group $G$, a simply
connected CW complex $X$ is said to be an $M(G,\,n)$-space if $$\widetilde{H}_i(X;\,\Z)=
	\begin{cases}
	G, &\text{for }\, i=n,\\
	0, &\text{otherwise}.
	\end{cases}
	$$
\end{definition}

Given $G$ and $n \,\geq\, 2$, any two Moore $M(G,\,n)$-spaces have the same homotopy type (cf. \cite[Example 4.34]{Hat}).
Consequently, for any $n\,\geq\, 2$, an $M(\Z,\,n)$-space is homotopy equivalent to the (real) $n$-sphere $S^n$. The following immediate consequence of it will be used later. 

\begin{obs}\label{An M(Z,n)-space is homotopic to an n-sphere}
Any $M(\Z,\,n)$-space, with $n\,\geq \,2$, is of rationally elliptic homotopy type.
\end{obs}

\begin{remark}
We will use the same definition of Moore $M(G,\,n)$-space for a complex analytic space $X$ which itself might not be a 
CW complex but is homotopically equivalent to a CW complex. Such an extension of the usual definition is possible because a homotopy 
equivalence preserves the reduced homology groups. For example, since a complex Stein space $X$ of 
complex dimension $n$ is homotopic to a CW complex of (real) dimension $n$ \cite[Proposition 2]{Ham},
we will say $X$ is a Moore $M(G,\,n)$-space if the homology 
groups of $X$ satisfy the same conditions as in Definition \ref{Def: Moore Spaces}.
\end{remark}

We will now state some well-known results on non-complete algebraic surfaces
which will be used here.

\subsection{Algebraic and analytic fibrations}

The reader is referred to \cite{GMM} for the definitions and basic properties stated below. By an \emph{algebraic fibration}, we always 
mean a surjective morphism of irreducible complex algebraic varieties $\varphi\,:\, X \,\longrightarrow\, Y$ such that the extension 
$\C(X)/\C(Y)$ of rational function fields is a regular extension, i.e., $\C(Y)$ is algebraically closed in $\C(X)$. Then Bertini's 
theorem, \cite[Theorem 2.26]{Sha1}, says that there is a non-empty Zariski-open subset $U\, \subset\, Y$ such that every fiber of 
$\varphi$ over $U$ is an irreducible algebraic variety. Any closed fiber $\varphi^{-1}(y)$ over $y\, \in\, U$ will be called a 
\emph{general fiber}. The morphism $\varphi$ is said to be an \emph{$F$-fibration} if the general closed fibers of $\varphi$ are 
isomorphic to $F$ as complex algebraic varieties. If $\varphi$ is an $F$-fibration, then every $y \,\in\, U$ has an open neighbourhood 
(in the usual analytic topology) $U_y$ such that $$\varphi^{-1}(U_y)\,\cong\, F \times U_y.$$ This is the local-triviality property of 
$\varphi$ around every closed point $y\,\in\, U$. In general, $\varphi$ may not be a smooth morphism, i.e., $F$ may not always be a 
smooth variety. If $F$ is smooth, a fiber $\varphi^{-1}(z)$ of $\varphi$ is called a \emph{singular fiber} if $\varphi^{-1}(z)$ is not 
isomorphic to $F$. In other words, a fiber over some point $z \,\in\, Y$ is a singular fiber of $\varphi$ if $\varphi$ is not 
locally-trivial around $z$ in the usual analytic topology. If for every point $y\,\in\, Y$, the local triviality for $\varphi$ around 
$y$ holds, we say that $\varphi$ is an \emph{$F$-bundle}. If $F$ is isomorphic to $\A^1$, $\A^1-\{0\}$ or 
$\PP^1$, then an $F$-fibration is called an \emph{$\A^1$-fibration}, \emph{$\A^1_\ast$-fibration}, or \emph{$\PP^1$-fibration} 
respectively.

Likewise, we define the above notions for a fibered surface in the analytic set-up. For a smooth connected complex analytic surface
$V$, and a smooth connected complex analytic curve $C$ (i.e., a Riemann surface), a surjective holomorphic map $$f\,:
\,V \,\longrightarrow \, C$$ is called an \emph{analytic fibration} if all the fibers of $f$ are of pure dimension $1$ and all but finitely many of the fibers are smooth and connected. Like in the algebraic situation described above, there is an analytic open subset 
$N$ of $C$ such that all fibers over points in $N$ are analytically isomorphic and around every point $p\,\in \,N$, the 
map $f$ is locally a topological fiber bundle. Such fibers are called \emph{general fibers} in this context. Here we will 
mostly consider analytic $\C$-fibrations and $\C^*$-fibrations.

Next, we recall the definition of \emph{total multiplicity} of a fiber of an algebraic fibration defined on an algebraic surface.

\begin{definition}
Let $\varphi\,:\, X\,\longrightarrow\, C$ be a surjective morphism from a smooth algebraic surface $X$ to a smooth algebraic
curve $C$. Let $F$ be a fiber of $\varphi$ having irreducible components $F_1,\, \cdots, \,F_r$ with
$r\,\in \,\N$, and let $m_i\,:=\,\length(\mathcal{O}_{F_i,\xi_i})$, where $\xi_i$ is the generic point of $F_i$,
be the \emph{multiplicity} of $F_i$. The irreducible component $F_i$ is reduced if and only if $m_i
\,= \,1$. Scheme-theoretically $F$ can be expressed as $$F\,=\,\sum\limits_{i=1}^{r}F_i\,=\,\sum\limits_{i=1}^{r}m_i\cdot(F_i)_{\textrm{red}}.$$
\begin{itemize}
\item The \emph{total multiplicity} (or, \emph{multiplicity}) of $F$ is $\mult(F)\,:=\,\gcd(m_1,\,\cdots, \,m_r)$. 
\item If $\mult(F)\,>\,1$, then $F$ is called a \emph{multiple fiber}. \item If $\mult(F)\,=\,1$, we call $F$ a \emph{non-multiple fiber}.
\end{itemize}
\end{definition}

\subsection{Singular fibers of a $\C^*$-fibration}

We now recall the description of any singular fiber of a $\C^*$-fibration $f\,:\,X\,\longrightarrow\, B$ from a smooth affine surface $X$ onto a smooth algebraic
curve $B$ (cf. \cite[Lemma 2.9]{Miy-Sug1}).
A singular fiber $F$ is of the form $\Gamma \sqcup \Delta$ such that:
\begin{enumerate}
\item $\Gamma$ is either empty, or $m\C^*$ with $m\,\in\, \N$, or it is $m_1A_1\cup m_2A_2$ with $\gcd(m_1,\,m_2)\,=\,1$, where $A_i$ is
isomorphic to $\A^1$ for $i\,=\,1,2$ and $A_1$, $A_2$ intersect transversally and intersect exactly at one point.

\item $\Delta$ is either empty, or a disjoint union of copies of $\A^1$ possibly occurring with multiplicities.
\end{enumerate}

\subsection{Ramified covering trick}

Ramified covering trick is used for eliminating 
multiple fibers of an algebraic fibration.

\begin{theorem}[{{\bf Ramified covering trick}; cf. \cite{Cha} and \cite[III, 3.2.1]{Miy2}}]\label{Ramified covering trick}
Let $f\,:\,X \,\longrightarrow\, C$ be a surjective morphism with an irreducible general fiber $F$, where $X$ is
a smooth irreducible algebraic surface and $C$ is a smooth irreducible quasi-projective curve. Let $p_1,\, \cdots,\, p_r$
be the points of $C$ such that $f^{\ast}(p_i)$ has total multiplicity $m_i\,>\,1$, and for all $p\,\in\, C\setminus
\{p_1,\, \cdots,\, p_r\}$ the total multiplicity of the fiber $f^{\ast}(p)$ is $1$. If $C \,\cong \,
\PP^1$, assume that either $r\, \geq\, 3$ or $r\,=\, 2$ with $m_1\,=\, m_2$.
Then there is a finite ramified Galois covering $\pi \,:\, D \,\longrightarrow\, C$ which is branched precisely over
$p_1, \,\cdots,\, p_r$ with the ramification index being $m_i$ over each $p_i$. Then the commutative
diagram
$$
\begin{tikzcd}[row sep=large, column sep=large]
\overline{X \times_{C} D} \arrow[r, "p_1"] \arrow[d, "p_2"'] &X \arrow[d, "f"]\\
D \arrow[r, "\pi"'] & C
\end{tikzcd}
$$
where $\overline{X \times_{C} D}$ is the normalized fiber product, has the following properties:
\begin{enumerate}
\item $\overline{X \times_{C} D}$ is smooth,

\item $p_1 \,:\, \overline{X \times_{C} D} \,\longrightarrow\, X$ is a finite Galois {\'e}tale covering of $X$, and

\item $p_2 \,: \,\overline{X \times_{C} D} \,\longrightarrow\, D$ does not have any multiple fiber.
\end{enumerate}
\end{theorem}

\subsection{Nori's lemma and its generalization}

M. V. Nori proved, in \cite{Nor}, an important result on a short exact sequence involving fundamental groups of smooth algebraic varieties.

\begin{theorem}[{{\bf Nori's lemma}; \cite[Lemma 1.5]{Nor}}]\label{Nori's Lemma}
Let $X$ and $Y$ be smooth connected complex algebraic varieties and $f\,:\, X
\,\longrightarrow\, Y$ a dominant morphism with a connected general fiber $F$. Assume that there is a closed subset
$S\, \subset\, Y$, of codimension at least two, 
outside which all the fibers of $f$ have at least one smooth point (i.e., $f^{-1}(p)$ is generically reduced on at least
one irreducible component of $f^{-1}(p)$ for all $p \,\in\, Y\setminus S$). 
Then the natural sequence of fundamental groups $$\pi_1(F)\,\longrightarrow\, \pi_1(X) \,\longrightarrow\,\pi_1(Y)\,\longrightarrow\, (1)$$
is exact.
\end{theorem}

See also \cite[Remark 3.13, 3.14]{GGH}, for the two important remarks related to the above lemma. These are used in the latter proofs.

The following result due to Xiao Gang is a useful generalization of Theorem \ref{Nori's Lemma}.

\begin{lemma}[{\cite[Lemma 2]{Gan}}]\label{Lem: Gang's Generalization}
Let $f \,:\, X\, \longrightarrow\, C$ be a surjective morphism from a smooth algebraic surface to a smooth algebraic curve
such that a general fiber $F$ of $f$ is smooth and irreducible. Let $\{p_1,\,\cdots,\,p_r\}\, \subset\, C$ be the image of the
multiple fibers of $f$ with the total multiplicity of the fiber over $p_i$ being $m_i\, >\, 1$, $1\, \leq\, i\, \leq\, r$.
Let $\wt{C}$ be the smooth completion of $C$. Let $\wt{C}\setminus C\,=\{p_{r+1},\,\cdots,\, p_{r+\ell}\}$ and
$g\, =\, {\rm genus}(\wt{C})$. Then there is an exact sequence
$$\pi_1(F)\, \longrightarrow\, \pi_1(X ) \, \longrightarrow\, \Gamma \, \longrightarrow\, (1),$$
where $\Gamma$ is the group with generators $$\alpha_1,\, \beta_1, \,\cdots, \,\alpha_g,\, \beta_g,\, \gamma_1,\,\cdots,
\, \gamma_r,\,\gamma_{r+1},\, \cdots,\, \gamma_{r+\ell}$$ and relations
$$[\alpha_1,\, \beta_1]\cdots [\alpha_g, \,\beta_g] \gamma_1\cdots \gamma_r \cdot \gamma_{r +1}\cdots \gamma_{r +\ell}
\,=\, 1 \,= \,\gamma_1^{m_1}\,=\, \cdots \,=\, \gamma_r^{m_r};$$
here $[a, \,b]\,=\, aba^{-1}b^{-1}$. If there are no multiple fibers in the $F$-fibration $f$, then
$\Gamma$ is the fundamental group of $C$.
\end{lemma}

\subsection{Some more well-known topological results}

Here we will mention some results which will be used in our later proofs.

\begin{theorem}[Hurewicz Theorem]\label{Hurewicz Theorem}
Let $X$ be a connected topological space with $\pi_i(X) \,=\, 0$ for all $0\,\leq\,i \,<\, n$, where
$n \,\geq\, 2$ is an integer. Then $\widetilde{H}_i(X;\,\Z) = 0$ for all $0\,\leq\,i \,< \,n$ and the Hurewicz homomorphism
$$h_n\,:\,\pi_n(X) \,\longrightarrow \, H_n(X;\,\Z)$$ is an isomorphism.
Moreover, if $X$ is a CW complex, then the Hurewicz homomorphism $$h_{n+1} \,: \,\pi_{n+1}(X)\,\longrightarrow\,H_{n+1}(X;\,\Z)$$ is also surjective.
\end{theorem}

For the above theorem, we refer to \cite[Chapter 8, Theorem 3]{Mos-Tan} and \cite[Exercise 23, Section 4.2]{Hat}.

\begin{theorem}[Rational Hurewicz Theorem]\label{Rational Hurewicz Theorem}
Let $X$ be a simply-connected topological space with $\pi_i(X) \otimes_{\Z}\Q\,=\, 0$ for all $2\,\leq\,i \,<\, n$. Then the Hurewicz homomorphism $h_i$ (as in Theorem \ref{Hurewicz Theorem}) induces an isomorphism
$$h_i\otimes \id_{\Q}\, :\, \pi_i(X) \otimes_{\Z}\Q \,\longrightarrow\, H_i (X;\, \Q)$$
for all $1 \,\leq\, i \,<\, 2n-1$ and a surjection for $i\, =\, 2n- 1$.
\end{theorem}

This is a well-known theorem in topology. We refer the reader to \cite[Corollary 3.4]{Dy} and \cite[Theorem 1.1]{Kla-Kre} for its proofs.

\begin{theorem}[{Whitehead's Theorem; \cite{Whi1}, \cite{Whi2}}]\label{Whitehead's Theorem}
If $f\,:\,X\,\longrightarrow\, Y$ is a map of CW complexes which induces isomorphisms on all homotopy groups, then $f$ is a homotopy equivalence.
\end{theorem}

\begin{theorem}[{Narasimhan's Theorem}]\label{Narasimhan's Theorem}
Let $X$ be an arbitrary complex Stein space (with or without singularities) of dimension $n$. Then the following statements hold:
\begin{enumerate}
\item $H_k(X,\,\Z)\,=\,0$\,\, for $k \,>\, n$;

\item $H_n(X,\,\Z)$\,\, is torsion-free;

\item $H^k(X,\,\Z)\,=\,0$\,\, for $k \,>\, n+1$;

\item $H^{n+1}(X,\,\Q)\,=\,0$.
\end{enumerate}
\end{theorem}

For Theorem \ref{Narasimhan's Theorem} we refer to \cite[Theorem 3]{Nar} and the corollary 
following it. If $X$ is a complex Stein manifold (i.e., without singularities) of dimension $n$, it was 
proved in \cite{And-Fra} that $H_k(X;\,\Z)$ are trivial for $k \,>\, n$ and $H_n(X; \,\Z)$ is free.

\section{Dichotomy for homotopy groups of Stein surfaces}\label{se3}
We will now discuss on elliptic-hyperbolic dichotomy for the higher homotopy groups of a finite CW complex. The following definition is in contrast
to what we have defined as rationally elliptic homotopy type in the Introduction.

\begin{definition}\label{Def:Hyperbolic Homotopy Type}	
Let $X$ be a simply connected CW-complex. Then $X$ is said to have \emph{rationally hyperbolic homotopy type} if
$\sum_{i=2}^{k}\dim \pi_i(X)\otimes_\Z \Q$ \emph{grows exponentially in $k$}, i.e., there exists a real number $C\,>\,1$ that depends only on $X$ such
that $$\sum_{i=2}^{k}\dim \pi_i(X)\otimes_\Z \Q \,>\, C^k$$ for all positive integers $k$ large enough.
\end{definition}
	
The higher homotopy groups of simply connected finite (hence compact) CW-complexes enjoy a dichotomy, which is recalled below.

\begin{theorem}[{\bf The elliptic-hyperbolic dichotomy for finite complexes}]\label{Thm: E-H dichotomy for finite complexes}
Let $X$ be a simply connected finite CW complex. Then either $X$ is of rationally elliptic homotopy type or it is of rationally
hyperbolic homotopy type.
\end{theorem}

For Theorem \ref{Thm: E-H dichotomy for finite complexes} we refer to \cite{Fel-Hal-Tho1} and \cite[\S~33]{Fel-Hal-Tho2}.

In \cite{Ham}, Helmut A. Hamm proved that any complex affine variety of dimension $n$ has the homotopy type of a finite CW complex of 
(real) dimension at most $n$. Therefore, from Theorem \ref{Thm: E-H dichotomy for finite complexes} it follows that any simply 
connected complex affine variety is either of elliptic homotopy type or of the hyperbolic homotopy type. The same holds for any complex 
affine variety having a finite fundamental group because any finite \'etale covering of an affine variety is again an affine variety 
and the higher homotopy groups remain unchanged.

In \cite[Proposition 2]{Ham}, Hamm had also shown that an $n$-dimensional complex Stein space is homotopic to a real $n$-dimensional CW 
complex. It may be noted that a Stein space may not always be homotopic to a finite CW complex. Thus Theorem \ref{Thm: E-H dichotomy 
for finite complexes} does not give a similar dichotomy result for simply connected Stein spaces. However the dichotomy --- in fact a 
stronger form of it --- does hold for Stein surfaces as shown below.

\begin{theorem}\label{Thm: Marked dichotomy for Stein surfaces}
Let $X$ be a complex connected Stein surface. Then the higher homotopy groups of $X$ satisfy one of the two mutually exclusive properties:
\begin{enumerate}
\item[(i)] $\sum_{i\geq2}\dim \pi_i(X)\otimes_\Z \Q \,\leq\, 2$;

\item[(ii)] $\sum_{i=2}^{k}\dim \pi_i(X)\otimes_\Z \Q$ grows exponentially in $k$.
\end{enumerate}
\end{theorem}

\begin{proof}
Let $\wt{X}$ be a universal covering of $X$. Since $X$ is a connected Stein space, so is $\wt{X}$. 

{\bf Case 1.}\, {\it Suppose that
$\sum_{i\geq2}\dim \pi_i(X)\otimes_\Z \Q$ is finite.}

Since $\pi_i(\wt{X})\,\cong\, \pi_i(X)$ for all $i\,\geq\, 2$ (cf. \cite[Chapter 4, Proposition 4.1]{Hat}), it follows that
the graded $\Q$-vector space $\pi_\ast(\wt{X};\, \Q)$ is finite-dimensional. In particular,
$\pi_2(X)\otimes_\Z\Q$ is a finite-dimensional $\Q$-vector space. By Theorem \ref{Hurewicz Theorem},
\begin{equation}\label{e-1}
H_2(\wt{X};\,\Z)\,\cong\, \pi_2(\wt{X})\,\cong\, \pi_2(X).
\end{equation}
So $H_2(\wt{X};\,\Q)$ is finite-dimensional. Thus by the universal coefficient theorem for cohomology (cf. \cite[Chapter 3, Theorem 3.2]{Hat}),
$$H^2(\wt{X};\,\Q)\,=\,\Hom_\Q(H_2(\wt{X};\, \Q),\,\Q)$$ is a finite-dimensional $\Q$-vector space.
Since $\wt{X}$ is a Stein surface, Theorem \ref{Narasimhan's Theorem} implies that $H^j(\wt{X};\, \Q)\,=\, 0$ for $j\, \geq\, 3$. Hence
the graded $\Q$-vector space $H^\ast(\wt{X};\, \Q)\,:=\, \bigoplus_{i\geq 0} H^i(\wt{X};\,\Q)$ is finite-dimensional. Now applying Friedlander--Halperin's result (cf. \cite[Corollary 1.3]{Fri-Hal}), it follows that
\begin{equation}\label{ej}
\sum_{i\geq 2}\dim \pi_i(X)\otimes_\Z \Q\, =\, \sum_{i\geq 2}\dim \pi_i(\wt{X})\otimes_\Z \Q
\,\leq\, n(\wt{X})\,\leq\, \dim{\wt{X}}\,=\,\dim X\,=\,2,
\end{equation}
where,
\begin{equation}\label{e0}
n(\wt{X})\,:=\,\sup\,\{m\,\in\, \Z\, \big\vert\,H^m(\wt{X};\,\Q)\,\neq\, 0\}.
\end{equation}
The last inequality in \eqref{ej} holds because, for any complex Stein space $V$, we have $n(V)\,\leq\, \dim V$ by Theorem \ref{Narasimhan's Theorem}.

{\bf Case 2.}\, {\it Assume that $\sum_{i\geq2}\dim \pi_i(X)\otimes_\Z \Q$ is infinite.}

If $\pi_2(X)\otimes_\Z \Q$ is infinite-dimensional, then the 
the second statement in the theorem is trivially satisfied by taking $C$ to be any real number $>\,1$. So assume that
$\pi_2(X)\otimes_\Z \Q$ is 
finite-dimensional, say
\begin{equation}\label{e1}
\dim \pi_2(X)\otimes_\Z \Q \,=\, k.
\end{equation}
Note that, $\pi_2(X)\,\cong\, H_2(\wt{X};\,\Z)$ using \eqref{e-1}, and $H_2(\wt{X};\,\Z)$ is a free abelian group since $\wt{X}$ is a Stein surface (cf. \cite[Korollar, Satz 1]{Ham}). Thus, $$\pi_2(X\,)=\,
	\begin{cases}
	\Z^k, &\text{if } \,k\,>\,0;\\
	0, &\text{if } \,k\,=\,0.
	\end{cases}
	$$

{\bf Claim.}\, \emph{The integer $k$ in \eqref{e1} is at least $2$.}

{\it Proof of Claim.} To prove this by contradiction, we first assume that $k=0$, i.e., $\pi_2(\wt{X})\,=\,0$. Since $\wt{X}$ is a Stein surface, $H_i(\wt{X};\,\Z)\,=\, 0$ for all $i\,\geq\, 3$ (cf. Theorem \ref{Narasimhan's Theorem}). Using Theorem \ref{Hurewicz Theorem} repeatedly, we obtain that all the homotopy groups of $\wt{X}$ are trivial. Now Theorem \ref{Whitehead's Theorem} applied to the constant map
from $\wt{X}$ to any chosen point $p \,\in\,\wt{X}$ proves that $\wt{X}$ is homotopic to $p$, whence $\wt{X}$ is contractible. But this is a contradiction because $\sum_{i\geq2}\dim\pi_i(X)\otimes_\Z \Q$ is infinite by assumption. Therefore, we conclude that $k$ is a positive integer.

Next assume that $k\,=\,1$. Then the reduced homology $\wt{H}_i(\wt{X};\,\Z)$ is trivial for all $i\, \not=\, 2$ and 
$H_2(\wt{X};\,\Z)\,\cong\, \Z$. 
Thus in this case, $\wt{X}$ is a Moore $M(\Z,\,2)$-space. Therefore, $\wt{X}$ is of the elliptic 
homotopy type by Observation \ref{An M(Z,n)-space is homotopic to an n-sphere}. But once again this leads to --- as above --- a
contradiction to the given condition that $\sum_{i\geq2}\dim \pi_i(X)\otimes_\Z \Q\,=\, \infty$. This implies $k\,\geq\, 2$ and the 
claim is proved.

Even if $k\,\geq\,2$, a very similar argument as in the case of $k\,=\,1$ in fact proves that $\wt{X}$ is a Moore $M(\Z^k,\, 2)$-space. Therefore $\wt{X}$ is homotopic to $S^2\,\vee \,S^2 \, \vee\,
\cdots\, \vee \,S^2$ (i.e., an 
wedge sum of $k$-copies of $S^2$), which is a simply connected finite CW complex. Since $\sum_{i\geq2}\dim \pi_i(X)\otimes_\Z \Q\,=\, \infty$, from Theorem \ref{Thm: E-H dichotomy for finite 
complexes}, it follows that $\wt{X}$ is of the hyperbolic homotopy type and hence so is $X$. 

This completes the proof.
\end{proof}

\section{Stein spaces of elliptic homotopy type}\label{se4}

In this section, we will investigate, for a complex connected Stein space $X$ of arbitrary dimension, the upper bound of
$\dim \pi_\ast(X;\, \Q)$ when it is finite.

\begin{theorem}\label{Thm: Bound of the rank of homotopy graded algebra for Stein spaces of dimension up to 3}
Let $X$ be a complex connected Stein space of dimension $\leq\,3$. Then $\pi_\ast(X;\, \Q)$
is finite-dimensional if and only if $$\dim \pi_\ast(X;\, \Q)\, :=\, \sum_{i\geq2}\dim \pi_i(X)\otimes_\Z \Q \,\leq\, \dim X.$$
\end{theorem}

\begin{proof}
A complex connected Stein space of dimension $1$ is clearly an Eilenberg--MacLane $K(\pi,\,1)$-space. So
if $\dim X\,\leq \,2$, the theorem follows from Theorem \ref{Thm: Marked dichotomy for Stein surfaces}.
Thus assume that $X$ is a connected Stein $3$-fold.

Let $\wt{X}$ be a universal cover of $X$. Assume that $\sum_{i\geq2}\dim \pi_i(\wt{X})\otimes_\Z \Q\, <\, \infty$.
In particular, $\pi_2(\wt{X})$ and $\pi_3(\wt{X})$ are of finite rank. By Theorem \ref{Hurewicz Theorem}, $H_2(\wt{X};\,\Q)\,\cong\, \pi_2(\wt{X})\otimes_\Z\Q$ is also finite-dimensional. Theorem \ref{Hurewicz Theorem} also implies that the Hurewicz homomorphism $$h_3\,:\,\pi_3(\wt{X}) \, \longrightarrow \, H_3(\wt{X};\,\Z)$$ is surjective and 
therefore $$h_3\otimes \id\,:\,\pi_3(\wt{X})\otimes_\Z\Q \, \longrightarrow \, H_3(\wt{X};\,\Q)$$ is also surjective. Since 
$\pi_3(\wt{X})\otimes_\Z\Q$ is finite-dimensional, it follows that $H_3(\wt{X};\,\Q)$ is finite-dimensional. By the universal 
coefficient theorem for cohomology (cf. \cite[Chapter 3, Theorem 3.2]{Hat}), $H^i(\wt{X};\,\Q)\,=\,\Hom_\Q(H_i(\wt{X};\,\Q),\,\Q)$ is 
finite-dimensional for $i\,\leq\, 3$. Again, since $\wt{X}$ is a Stein $3$-fold, $H^j(\wt{X};\,\Q)\,=\,0$ for all $j>3$ using Theorem \ref{Narasimhan's Theorem}. Therefore, $H^*(\wt{X};\,\Q)$ is a finite-dimensional 
$\Q$-vector space. Since $\pi_*(\wt{X};\,\Q)$ is also finite-dimensional by assumption,
Friedlander-Halperin's result (cf. \cite[Corollary 1.3]{Fri-Hal}) completes the proof.
\end{proof}

In view of Theorem \ref{Thm: Bound of the rank of homotopy graded algebra for Stein spaces of dimension up to 3}, we
ask the following question.

\begin{question}
Let $X$ be a complex connected Stein $3$-fold. Then do the higher homotopy groups of $X$ satisfy one of the
following two mutually exclusive properties?
\begin{enumerate}
\item[(i)] $\sum_{i\geq2}\dim \pi_i(X)\otimes_\Z \Q \,\leq\, 3$;

\item[(ii)] $\sum_{i=2}^{k}\dim \pi_i(X)\otimes_\Z \Q$ grows exponentially in $k$.
\end{enumerate}
\end{question}

The following theorem on a class of complex connected Stein spaces of dimension at least $4$ is rather similar to
Theorem \ref{Thm: Bound of the rank of homotopy graded algebra for Stein spaces of dimension up to 3}.

\begin{theorem}\label{Thm: Bound of the rank of homotopy graded algebra for certain Stein spaces of dimension at least 4}
Let $X$ be a complex connected Stein space of dimension at least $4$. Then the following assertions hold. 

\begin{enumerate}
\item Assume that 
\begin{equation}\label{van_hom}
\pi_i(X)\otimes_\Z \Q \,=\,0, \quad \forall\ \ \, 2\,\leq\,i\,\leq\, \bigg\lfloor \frac{\dim X}{2} \bigg\rfloor,
\end{equation}
where $\lfloor\,.\,\rfloor$ denotes the greatest integer function. Then $X$ satisfies the finite homotopy rank-sum property if and only if 
$$\sum_{i\geq2}\dim \pi_i(X)\otimes_\Z \Q \,\leq\, \dim X.$$
\item Moreover, if $\dim X\,=\,4$ and $\pi_2(X)\otimes_\Z\Q\,=\,0$, then $X$ satisfies the finite homotopy rank-sum property
if and only if $$\sum_{i\geq2}\dim \pi_i(X)\otimes_\Z \Q \,\leq\, 3.$$
\end{enumerate}
\end{theorem}

\begin{proof}
Let $\wt{X}$ be a universal cover of $X$. Assume that
$\pi_\ast(X;\,\Q)$ is finite-dimensional. Thus $\dim \pi_i(\wt{X})\otimes_\Z\Q\, <\, \infty$ for all $i\,>\,1$.

\medskip
{\it Proof of (1).}\, Note that it suffices to prove that $H^*(\wt{X};\,\Q)$ is finite-dimensional.
Indeed, in that case Friedlander-Halperin's result (cf. \cite[Corollary 1.3]{Fri-Hal}) completes the proof exactly
in the same way as in the proof of Theorem \ref{Thm: Bound of the rank of homotopy graded algebra for Stein spaces of dimension up to 3}.

To prove that $H^*(\wt{X};\,\Q)$ is finite-dimensional, we use \cite[Theorem 1.1]{Kla-Kre}, which says that for a simply connected topological space $Y$ and an
integer $r \,>\,2$ satisfying $\pi_i(Y)\otimes_\Z\Q \,=\, 0$ for all $1 \,<\, i\, <\, r$, the Hurewicz homomorphism
induces an isomorphism $$h_i\otimes \id \,:\, \pi_i(Y)\otimes_\Z\Q \,\longrightarrow\,H_i(Y;\,\Q)$$ for
$1 \,\leq\, i\, <\, 2r\,-\,1$ and a surjection for $i \,=\, 2r\,-\,1$. Therefore, using \eqref{van_hom}
we obtain that $H_i(\wt{X};\,\Q)$ is finite-dimensional for $1\,\leq\,i\,\leq\,2\big\lfloor \frac{\dim X}{2} \big\rfloor\,
+\,1$. Now it is clear that $$2\bigg\lfloor \frac{\dim X}{2} \bigg\rfloor\,+\,1 \,\geq\, \dim X.$$ Therefore, it follows that
$\dim H_i(\wt{X};\,\Q)\, <\, \infty$ for all $1\,\leq\,i\,\leq\,\dim X$. Thus by the universal coefficient theorem for
cohomology (cf. \cite[Chapter 3, Theorem 3.2]{Hat}), $H^i(\wt{X};\,\Q)$ is also finite-dimensional for
all $1\,\leq\,i\,\leq\,\dim X$. Since $X$ is connected Stein space, so is $\wt{X}$ and thus by Theorem \ref{Narasimhan's Theorem} we have $H^j(\wt{X};\,\Q)\,=\,0$ for all $j \,>\,\dim X$. Hence $H^*(\wt{X};\,\Q)$ is finite-dimensional. This completes the proof of the first part.

\medskip
{\it Proof of (2).}\, Whenever $\dim X\,=\,4$, the condition in \eqref{van_hom} is equivalent to the statement
that $\pi_2(X)\otimes_\Z\Q\,=\,0$. Using the proof of (1), it is evident that $\pi_\ast(\wt{X};\, \Q)$ and $H^\ast(\wt{X};\,\Q)
\,:=\, \bigoplus_{i\geq 0} H^i(\wt{X};\,\Q)$ are both finite-dimensional $\Q$-vector spaces. We define $$x_k\,:=\,\rank(\pi_k(\wt{X}))\,=\,\rank(\pi_k(X))\,=\,\dim \pi_k(X)\otimes_\Z\Q,$$ for all integers $k\,\geq\,2$.
Since $x_2\,=\,0$, using Friedlander--Halperin's result (cf. \cite[Corollary 1.3]{Fri-Hal}) it can be deduced that 
\begin{equation}\label{c1}
\sum_{k\geq 2}2kx_{2k}\,\,\leq\,\, 4,
\end{equation}
\begin{equation}\label{c2}
\sum_{k\geq 2}(2k-1)x_{2k-1}\,\,\leq\,\, 7.
\end{equation}
Indeed, we know that $$\sup\,\{d\,\in\, \Z\,\,\big\vert\,\, H^d(\wt{X};\, \Q)\,\neq\,0\} \,\leq\,4$$ by applying Theorem \ref{Narasimhan's Theorem} to the Stein $4$-fold $\wt{X}$, and
hence \eqref{c1} and \eqref{c2} follow. It is deduced from \eqref{c1} and \eqref{c2} that
\begin{itemize}
\item $x_4\,\leq\,1$,

\item $x_{2i}\,=\,0\,=\,x_{2j-1}$ for all $i\,>\,2$ and $j\,>\,4$, and 

\item the inequality
\begin{equation}\label{c3}
3x_3\,+\,5x_5\,+\,7x_7\,\leq\,7
\end{equation}
holds.
\end{itemize}

{\bf Claim.}\, \emph{$x_5\,=\,0$.}

{\it Proof of Claim.}\, If $x_5\,\not=\,0$, then from \eqref{c3}, it is evident that $x_3\,=\,0$. Since $x_k\,=\,0$ for $1\,<k\,<4$, 
applying Theorem \ref{Rational Hurewicz Theorem}, it is deduced that $x_5\,=\,\dim H_5(\wt{X};\,\Q)$, 
which is zero by Theorem \ref{Narasimhan's Theorem} as $\wt{X}$ is a Stein $4$-fold. This 
contradicts our assumption and hence the claim follows.

Now from \eqref{c3}, it is clear that $x_3 \,\leq\,2$ and $x_7\,\leq\,1$.

 {\bf Case 1.}\, \emph{Assume that $x_7\,=\,0$.}

In this case, it turns out that $x_k\,=\,0$ for all $k\,\geq\,5$; also recall that $x_3\,\leq\,2$ and $x_4\,\leq\,1$. Therefore 
$$\sum_{i\geq2}\dim \pi_i(X)\otimes_\Z \Q \,=\,x_3\,+\,x_4 \,\leq\, 3.$$

 {\bf Case 2.}\, \emph{Assume that $x_7\,=\,1$.}

In this case, it is clear from \eqref{c3} that $x_3\,=\,x_5\,=\,0$. So
$$\sum_{i\geq2}\dim \pi_i(X)\otimes_\Z \Q \,=\,x_4\,+\,x_7\,\leq\, 2.$$
Combining all cases together, the proof of the second part follows.
This completes the proof of the theorem.
\end{proof}

\section{Stein surfaces of elliptic homotopy type}\label{se5}

The following theorem describes all simply connected Stein surfaces of elliptic homotopy type.

\begin{theorem}\label{Characterization of Elliptic Homotopy Type Surfaces}
Let $X$ be a simply connected complex Stein surface. Then $X$ is of elliptic homotopy type if and only if one of the following
two statements hold:
\begin{enumerate}
\item[(i)] $X$ is contractible.

\item[(ii)] $X$ is homotopic to $S^2$; {\underline{\rm or}} equivalently, $H_2(X;\,\Z)\,=\,\pi_2(X)\,=\,\Z$.
\end{enumerate}
\end{theorem}

\begin{proof}
If $X$ satisfies one of the above two conditions, then $X$ is evidently of the elliptic homotopy type.
Concerning the second statement in the theorem, the equivalence can be established as follows.
Take a simply connected Stein surface $V$ with $H_2(V;\,\Z)\,=\,\Z$. By Theorem \ref{Narasimhan's Theorem} we have $H_j(V;\,\Z)\,=\,0$ for all $j \,\geq\, 3$. Therefore $V$ is a Moore $M(\Z,\,2)$-space,
and hence $V$ is homotopic to $S^2$.

To prove the converse, let $X$ be a simply connected complex Stein surface
of elliptic homotopy type. This implies, in particular, that $\pi_2(X)$ is of finite rank. 
Since by Theorem \ref{Hurewicz Theorem}, $\pi_2(X)\,\cong\, H_2(X;\,\Z)$, it follows that $H_2(X;\,\Q)$ is a 
finite-dimensional $\Q$-vector space. Thus $H^2(X;\,\Q)$ is also finite-dimensional. As $H^j(X;\,\Q)\,=\,0$ for all $j\,\geq\, 3$, from 
Theorem \ref{Narasimhan's Theorem} it follows that the graded $\Q$-vector space $H^\ast(X;\,\Q)$ is finite-dimensional. Therefore, 
using Friedlander--Halperin's result, \cite[Corollary 1.3]{Fri-Hal}, it is deduced that
\begin{equation}\label{e3}
\sum_{k\geq 1}2k\dim \pi_{2k}(X)\otimes_\Z\Q\,\,\leq\,\, n(X), \end{equation}
\begin{equation}\label{e4}
\sum_{k\geq 2}(2k-1)\dim \pi_{2k-1}(X)\otimes_\Z\Q\,\,\leq\,\, 2n(X)-1,
\end{equation}
where, $n(X)\,:=\,\sup\,\{d\,\in\, \Z\,\,\big\vert\,\, H^d(X;\, \Q)\,\neq\,0\}$ (see also \eqref{e0}).

Since $X$ is a Stein surface, we have $n(X)\,\leq\, 2$ by Theorem \ref{Narasimhan's Theorem}. Therefore, using \eqref{e3} and \eqref{e4}
it follows that 
\begin{equation}\label{eqn:e4'}
\pi_i(X)\otimes_\Z\Q\,=\,0 \text{ for all } i\,\geq\, 4 \text{ and } \dim \pi_j(X)\otimes_\Z\Q \,\leq\, 1 \text{ for } j\,=\,2,\,3. 
\end{equation}	
Suppose that $X$ is not contractible. Since $X$ is a Stein surface, by Theorem \ref{Narasimhan's Theorem}, it follows that $H_j(X;\,\Z)\,=\,0$ for all $j\,>\,2$. Since $X$ is simply
connected but not contractible, $H_2(X;\,\Z)\,\not=\, 0$. Theorem \ref{Hurewicz Theorem} says that $\pi_2(X)\,\cong\, H_2(X;\,\Z)\, \not=\, 0$. It is known that $H_2(X;\, \Z)$ is a free abelian group \cite[Korollar, Satz 1]{Ham}.
Thus $\pi_2(X)$ is also a finitely generated nontrivial free abelian group. Since $\dim \pi_2(X)\otimes_\Z\Q
\,\leq\, 1$ (by \eqref{eqn:e4'}), 
and $\pi_2(X)$ is non-trivial, we conclude that $\pi_2(X)\,\cong \,\Z$. Therefore $\widetilde{H}_i(X;\,\Z)$ is trivial
for all $i$ except $i\,=\,2$ and $H_2(X;\,\Z)\,\cong \,\Z$.
Consequently, $X$ is a Moore $M(\Z,\,2)$-space, which implies that $X$ is homotopically equivalent to $S^2$.
\end{proof}

The following corollary shows that for a complex connected Stein surface the finite homotopy rank-sum property depends only on its second homotopy group.

\begin{cor}\label{Cor: Characterization of E-type Stein surfaces}
Let $X$ be a complex connected Stein surface. Then $X$ satisfies the finite homotopy rank-sum property if and only if one of the
following two statements holds:
\begin{enumerate}
\item[(i)] $\pi_2(X)\,=\,0$; {\underline{\rm or}} equivalently, $X$ is an Eilenberg--MacLane $K(\pi_1(X),\,1)$-space.

\item[(ii)] $\pi_2(X)\,=\,\Z$; {\underline{\rm or}} equivalently the universal cover of $X$ is homotopic to $S^2$.
\end{enumerate}
\end{cor}

\begin{proof}
Let $\wt{X}$ be a universal cover of $X$. Since $X$ is a connected Stein surface, so is $\wt{X}$. Clearly, we have
$\sum_{i\geq2}\dim \pi_i(X)\otimes_\Z \Q \, <\, \infty$ if and only if $\wt{X}$ is of elliptic homotopy type. By
Theorem \ref{Characterization of Elliptic Homotopy Type Surfaces}, this is in fact equivalent to
the statement that $\wt{X}$ is either contractible or it is homotopic to $S^2$. Again, it is evident from the
definition that $\wt{X}$ is contractible if and only if $X$ is an Eilenberg--MacLane $K(\pi_1(X),\,1)$-space. Now it
only remains to establish the equivalences between two statements in both (i) and (ii).

By Theorem \ref{Hurewicz Theorem} it follows that $$H_2(\wt{X};\,\Z)\,\cong\, \pi_2(\wt{X})\,\cong\,\pi_2(X).$$ Again, since $\wt{X}$ is a Stein surface, $H_j(\wt{X};\,\Z)\,=\,0$ for $j\geq 3$ by Theorem \ref{Narasimhan's Theorem}. Thus $\wt{X}$ is a Moore $M(\pi_2(X),\,2)$-space. Therefore, the following are true:
\begin{itemize}
\item[(a)] $\wt{X}$ is contractible (equivalently, $X$ is an Eilenberg--MacLane $K(\pi_1(X),\,1)$-space) whenever $\pi_2(X)$ is trivial. 

\item[(b)] $\wt{X}$ is homotopic to $S^2$ whenever $\pi_2(X)$ is isomorphic to $\Z$.
\end{itemize}
The proofs of the other directions of the above equivalences in (i) and (ii) are straightforward. This completes the proof.
\end{proof}

S. Kaliman and M. Zaidenberg has introduced the technique of affine modification in \cite{Kal-Zai}. They have also deduced many topological results for certain 
hypersurfaces in affine spaces using the above technique See \cite[Section 4]{Kal-Zai} for more details. We will apply their result followed by Theorem 
\ref{Characterization of Elliptic Homotopy Type Surfaces} to construct smooth affine fourfolds of elliptic homotopy type. The following example is such an 
illustration.

\begin{example}
\emph{Let $X_0$ be an irreducible smooth affine hypersurface in $\C^3$ given by the equation $p(x,\,y,\,z)\,=\,0$,
where $p(X,\,Y,\,Z)\,\in \,\C[X,\,Y,\,Z]$ is a 
non-constant polynomial. Assume that $X_0$ is of the elliptic homotopy type. Then the affine hypersurface $X$ in $\C^5$ defined by the 
equation $uv\,=\,p(x,y,z)$ is a smooth fourfold which is either diffeomorphic to $\R^8$ or homotopically equivalent to the real $4$-sphere $S^4$. Consequently, $X$ is of the elliptic homotopy type.}

--- We will write a proof in the support of all the assertions made above.

--- Using the Jacobian criterion of smoothness it is easy to see that $X$ is smooth if and only if $X_0$ is smooth. Since $X_0$ is
a smooth and irreducible surface, $X$ is clearly an affine irreducible fourfold. Using \cite[Proposition 4.1]{Kal-Zai}, we get the following:
\begin{enumerate}
\item $X$ is simply connected, and

\item $\wt{H}_j(X;\,\Z) \,\cong\, \wt{H}_{j-2}(X_0;\,\Z)$ for all $j\,\geq\, 0$.
\end{enumerate}
Since $X_0$ is of the elliptic homotopy type, first of all, $X_0$ is simply connected, by definition. It follows from
Theorem \ref{Characterization of Elliptic Homotopy Type Surfaces} that $X_0$ is either contractible or it is homotopic to $S^2$.

{\bf Claim 1.}\, \emph{If $X_0$ is contractible, then $X$ is also contractible. In that case, $X$ is in fact diffeomorphic to $\R^8$.} 

{\it Proof of Claim 1.}\, Since $\wt{H}_j(X_0,\,\Z)\,=\,0$ for all $j \,\geq\, 0$, from \cite[Proposition 4.1]{Kal-Zai} it is 
deduced that $\wt{H}_j(X;\,\Z)\,=\,0$ for all $j\,\geq\, 0$. As $X$ is simply connected, and all its reduced integral 
homology groups are trivial, using Theorem \ref{Hurewicz Theorem} repeatedly it follows that $\pi_i(X)$ is 
trivial for all $i\,\in\, \N$. 
Thus $X$ is contractible. Now the assertion that $X$ is diffeomorphic to $\R^8$ follows immediately from \cite[Corollary 4.1]{Kal-Zai}.

{\bf Claim 2.}\, \emph{If $X_0$ is homotopic to $S^2$, then $X$ is homotopic to $S^4$.}

{\it Proof of Claim 2.}\, As $X$ is irreducible, we have $\wt{H}_0(X;\,\Z)\,=\,0$.
Since $X$ is simply connected it follows that $H_1(X;\,\Z)\,=\,0$. Again
the same result implies that
\begin{itemize}
\item $H_2(X;\,\Z)\,\cong\, \wt{H}_0(X_0,\,\Z)\,=\,\wt{H}_0(S^2,\,\Z)\,=\,0$;

\item $H_3(X;\,\Z)\,\cong\, H_1(X_0;\,\Z)\,=\,H_1(S^2;\,\Z)\,=\,0$; and

\item $H_4(X;\,\Z)\,\cong\, H_2(X_0;\,\Z)\,\cong\,H_2(S^2;\,\Z)\,\cong\, \Z$.
\end{itemize}
Next, since $X$ is an affine fourfold, we have $H_j(X;\,\Z)\,=\,0$ for all $j\,>\,4$ by Theorem \ref{Narasimhan's Theorem}.
Hence it follows that $X$ is a simply connected affine fourfold which is a Moore $M(\Z,\,4)$-space. Hence $X$ is homotopically equivalent to $S^4$.
\end{example}

The next theorem puts a strong necessary condition on the fundamental group of a complex connected Stein space $X$ for which
$\sum_{i\geq 2}\dim \pi_i(X)\otimes_\Z \Q$ is finite.

\begin{theorem}\label{Fundamental Group of Elliptic Type Stein Surface}
Let $X$ be a complex connected Stein surface satisfying the finite homotopy rank-sum property. Then the following two statements hold:
\begin{enumerate}
\item Every non-trivial torsion element of $\pi_1(X)$, if it exists, is of order $2$.

\item Let $\wt{X}$ be a universal cover of $X$. For every non-trivial torsion element $\sigma$ in $\pi_1(X)$ the quotient space 
$\wt{X}/{\langle \sigma \rangle}$ is a Moore $M(\Z/2\Z,\,1)$-space.
\end{enumerate}
\end{theorem}

\begin{proof}
Since $X$ is a Stein space, so is its universal cover $\wt{X}$. As $X$ satisfies the finite homotopy rank-sum property, we know
that $\wt{X}$ is of elliptic homotopy type. Therefore from Theorem \ref{Characterization of Elliptic Homotopy Type Surfaces} it 
follows that either $\wt{X}$ is contractible or $\wt{X}$ is homotopic to $S^2$.

{\it Proof of (1).}\, Suppose $\pi_1(X)$ contains a non-trivial torsion element
$\sigma$. Consider the action of the
cyclic subgroup $H\,:=\,\langle \sigma \rangle\, \subset\, \pi_1(X)$ on $\wt{X}$, and let
$$p\,:\, \wt{X}\,\longrightarrow\,\wt{X}/H$$ be the corresponding quotient map. Since $H$ is a finite group, this complex
analytic map $p$ is finite. Therefore, the induced homomorphism
\begin{equation}\label{ep}
p_{\ast}\,\,:\,\, H_i(\wt{X};\, \Q)\,\longrightarrow\, H_i(\wt{X}/H;\,\Q)
\end{equation}
is surjective for all $i\,>\,0$ (cf. \cite{Gie}).

As the action of $H$ on $\wt{X}$ is free, $p$ is a covering map of degree $|H|$, and hence 
\begin{equation}\label{eqn:covering euler char}
e(\wt{X})\,=\, (\deg p)\cdot e(\wt{X}/H)\,=\, |H|\cdot e(\wt{X}/H),
\end{equation}
where $e(M)$ denotes the Euler characteristic of $M$.

If $\wt{X}$ is contractible then \eqref{ep} implies that $H_i(\wt{X}/H;\,\Q)\,=\, 0$ for all $i\,>\,0$, and hence we have $$e(\wt{X}/H)\,=\,e(\wt{X})\,=\,1.$$ So \eqref{eqn:covering 
euler char} implies that $|H|\,=\,1$. But this contradicts the fact that $\sigma$ is a non-trivial element of $\pi_1(X)$. Thus $\wt{X}$ is homotopic to $S^2$ whence $H_2(\wt{X};\,\Q)\,=\,\Q$ and $H_j(\wt{X};\,\Q)\,=\,0$ for $j \, >\,2$. Note that $e(\wt{X})\,=\,2$. Now \eqref{ep} implies that $b_2(\wt{X}/H)\,\leq\, 1$ and $H_j(\wt{X}/H;\,\Q)\,=\, 0$ for $j \, >\,2$; i.e. $e(\wt{X}/H)\,\leq \, 2$. 
Since $\sigma$ is a non-trivial torsion element, we have $|H|\,>\,1$.
As $e(\wt{X})\,=\,2$, from \eqref{eqn:covering euler char} it follows that $|H|\,=\, \text{order}(\sigma) \,=\,2$ and
$e(\wt{X}/H)\,=\,1$. This completes the proof of (1).

{\it Proof of (2).}\, Let $\sigma$ be a non-trivial torsion element in $\pi_1(X)$ and $H\,:=\,\langle\sigma\rangle \,\subset\,\pi_1(X)$. Consider the universal covering map $p\,:\,\wt{X}\,\longrightarrow\,\wt{X}/H$ as above. Using covering space theory, it follows that $\pi_1(\wt{X}/H)
\,=\,H$. Using the first part, it turns out that $H\,\cong\,\Z/2\Z$. Thus $b_1(\wt{X}/H)\,=\,0$. Since $H_j(\wt{X}/H;\,\Q)\,=\, 0$ for $j \, >\,2$ and $e(\wt{X}/H)\,=\,1$, it follows that 
\begin{equation}\label{f2}
b_2(\wt{X}/H)\,=\,e(\wt{X}/H)\,+\,b_1(\wt{X}/H)\,-\,1\,=\,0.
\end{equation}
Since $\wt{X}$ is a Stein surface, so is its quotient $\wt{X}/H$. Therefore, we note that $H_2(\wt{X}/H;\, \Z)$ is free (cf. \cite[Korollar, Satz 1]{Ham}). Hence from 
\eqref{f2} we conclude that $H_2(\wt{X}/H;\,\Z)\,=\,0$. As $\wt{X}/H$ is a Stein surface, it is homotopic to a CW complex of real dimension at most $2$ \cite[Proposition 2]{Ham}. This implies that $\wt{X}/H$ is a Moore $M(\Z/2\Z,\,1)$-space. This completes the proof of (2).
\end{proof}

\begin{cor}\label{Cor:Finite Fundamental Group of Elliptic Type Affine Surface}
Let $X$ be a complex affine surface satisfying the finite homotopy rank-sum property. Assume that $X$ is not
simply connected, and every element of $\pi_1(X)$ has finite order. Then $\pi_1(X)\,\cong\,\Z/2\Z$ with $e(X)\,=\,1$.
\end{cor}

\begin{proof}
Since all the elements of $\pi_1(X)$ are of finite order, by Theorem
\ref{Fundamental Group of Elliptic Type Stein Surface} it follows that every element of $\pi_1(X)$ is of order $2$. Therefore,
the fundamental group $\pi_1(X)$ is abelian. 
Since $X$ is an affine variety, it follows that $\pi_1(X)$ is finitely generated. So, using the structure theorem of a finitely generated abelian group, it is deduced that
$\pi_1(X)$ is finite and $$\pi_1(X)\,\cong\, \underbrace{\Z/2\Z \times \Z/2\Z \times \cdots \times \Z/2\Z}_{n\text{-times}}.$$
Thus,
\begin{equation}\label{h2}
|\pi_1(X)|\ =\ 2^n
\end{equation}
for some $n \,\in\, \N$.

Let $\wt{X}$ be a universal cover of $X$. Since $X$ satisfies the finite homotopy rank-sum property, Corollary \ref{Cor: Characterization of 
E-type Stein surfaces} says that either $X$ is an Eilenberg--MacLane $K(\pi_1(X),\,1)$-space or $\wt{X}$ is homotopic to $S^2$. 
Since, by assumption, $\pi_1(X)$ contains a non-trivial torsion element, $X$ cannot be an Eilenberg--MacLane $K(\pi_1(X),\,1)$-space, 
as the fundamental group of an Eilenberg-MacLane $K(\pi,\,1)$ finite CW complex is torsion-free (cf. \cite[Proposition 2.45]{Hat}). Therefore $\wt{X}$ is homotopic to $S^2$, so $e(\wt{X})\,=\,e(S^2)\,=\,2$. Since $\pi_1(X)$ is finite, the 
following holds:
\begin{equation}\label{g}
e(\wt{X})\,=\,|\pi_1(X)|\cdot e(X).
\end{equation}
Now from \eqref{g} and \eqref{h2} we get that $2^ne(X)\,=\,2$.
Consequently, $n\,=\,1$, i.e., $\pi_1(X)\,\cong\,\Z/2\Z$, and $e(X)\,=\,1$. This completes the proof.
\end{proof}
	
\section{Affine surfaces with\, $\Kbar\leq 1$ of elliptic homotopy type}\label{se6}

In this section, we will prove that any smooth affine surface $X$ of non-general type (meaning
its logarithmic Kodaira dimension is at most one) satisfying the finite homotopy rank-sum property is in fact an Eilenberg-MacLane $K(\pi_1(X),\,1)$-space. 
A recent work, \cite{GGH}, classified all smooth affine Eilenberg--MacLane $K(\pi,\,1)$-surfaces of 
non-general type. This classification result is the following.

\begin{theorem}[{\cite[Theorem 5.3, Theorem 5.5, Theorem 5.9]{GGH}}]\label{Result : Classification}
Let $X$ be a non-contractible smooth complex affine $K(G,\,1)$ surface. Then the following statements hold:
\begin{enumerate}
\item If $\Kbar(X)\,=\,-\infty$, then $X$ is an $\A^1$-bundle over a smooth algebraic curve which is not isomorphic to
either $\A^1$ or $\PP^1$.

\item If $\Kbar(X)\,=\,0$, then either $X\,\cong\, \C^*\times \C^*$ or $X \,\cong\, H[-1,\,0,\,-1]$ (described below). 

Moreover, if $X \,\cong\, H[-1,\,0,\,-1]$, then there is a $2$-fold unramified covering of $X$ which is isomorphic to $\C^* \times \C^*$.

\item If $\Kbar(X)\,=\,1$, then $X$ admits a $\C^*$-fibration with all fibers isomorphic to $\C^*$, if taken with
the reduced structure. This
implies that there is a finite unramified covering $\widehat X$ of $X$ such that
$\widehat X$ is a $\C^*$-bundle over a smooth algebraic curve not isomorphic to $\PP^1$.
\end{enumerate}
\end{theorem}

\begin{definition}[{Fujita's $H[-1,0,-1]$; cf. \cite[\S 8.5]{Fuj2}}]\label{Def: Fujita's H[-1,0,-1]}\mbox{}
Start with $V_0\,=\,\mathbf{F}_n$ $(n\,\geq\, 1)$, the Hirzebruch surface of degree $n$. Let $M_n$ and $\ol{M}_n$ respectively denote the minimal section of $\mathbf{F}_n$ and a section of the ruling on $\mathbf{F}_n$ with $(\ol{M}_n\,\cdot\, M_n)\,=\,0$.
		
		Let $\ell_0$, $\ell_1$ and $\ell_2$ be three distinct fibers of the ruling on $V_0$. Put $p_i\,:=\,\ell_i \cap \ol{M}_n$, for $i\,=\,1,2$. 
		\begin{figure}[h!]
			$	\begin{tikzcd}[row sep=huge, column sep=huge]
			\FujitahGrid{} & \Fujitamiddle{} \arrow[l, "(\text{blow up})", "\alpha"'] \\
			& \Fujitah{} \arrow[u, "(\text{blow up})", "\beta"']
			\end{tikzcd}
			$
			\caption{Fujita's $H[-1,0,-1]$}
			\label{fig:Fujita's $H[-1,0,-1]$}
		\end{figure}
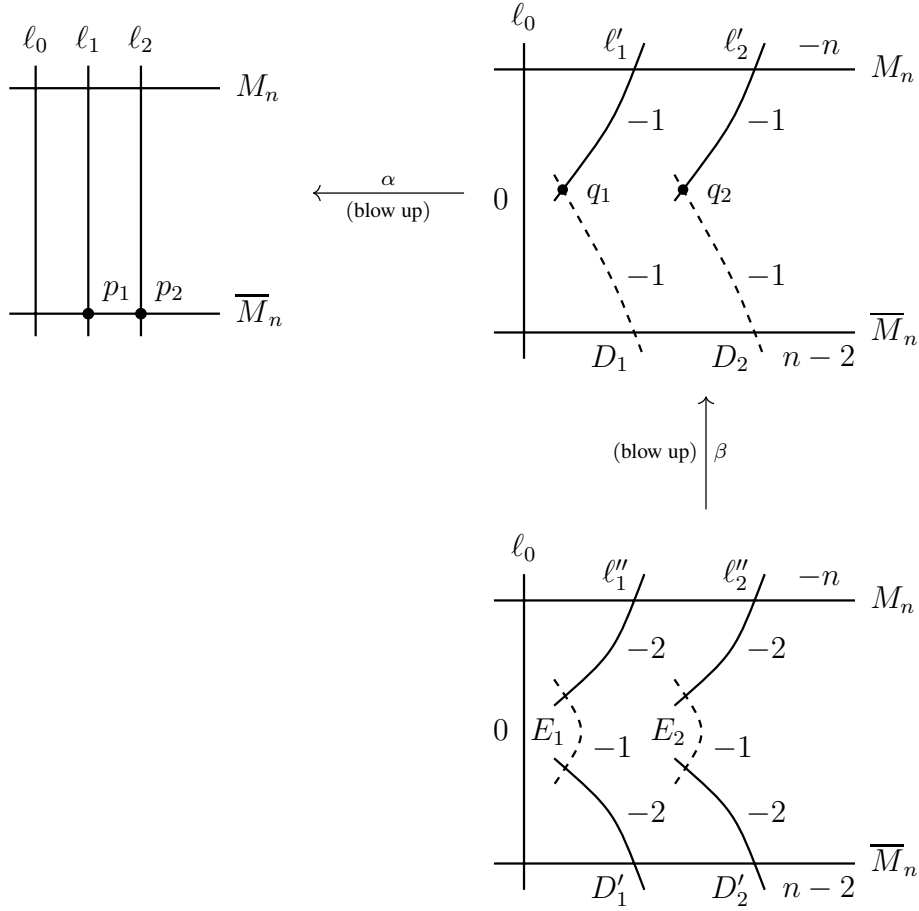
	
		Let $\alpha\,:\, V_1 \,\longrightarrow\, V_0$ be the blowing-up centering at $p_1$ and $p_2$. Let $D_i\,:=\,\alpha^{-1}(p_i)$ and assume $\ell_i'$ denote the proper transform of $\ell_i$ under $\alpha$, for $i\,=\,1,2$. Put $q_i\,:=\,\ell_i' \cap D_i$, for $i\,=\,1,2$. Let $\beta\,:\,V_2\,\longrightarrow\, V_1$ be the blowing-up centering at $q_1$ and $q_2$. Let $E_i\,:=\,\beta^{-1}(q_i)$. Assume that $\ell_i''$ and $D_i'$ respectively denotes the proper transform of $\ell_i'$ and $D_i$ under $\beta$, for $i\,=\,1,2$.
		Let $$D\ :=\ \ell_0\,+\,\ell_1''\,+\,\ell_2''\,+\,M_n\,+\,\ol{M}_n\,+\,D_1'\,+\,D_2',$$i.e., $D$ is a union of the bold-faced curves in $V_2$ in Figure \ref{fig:Fujita's $H[-1,0,-1]$}. Now the affine surface $Y\,:=\,V_2\,\setminus\,D$ is the surface $H[-1,0,-1]$ named by Fujita. 

It can be observed easily that $Y$ has an induced untwisted $\C^*$-fibration over a curve $B$ which is obtained from $\PP^1$ by removing one point, i.e., $B$ is isomorphic to $\A^1$.
\end{definition}

\subsection{Some preparatory results}

\begin{definition}\label{Def: Cross}
Let $C$ be an affine reducible curve having two irreducible components $D_1$ and $D_2$ meeting each other transversally at a point such
that, scheme-theoretically, $D_1$ is isomorphic to $m_1\A^1$, while $D_2$ is isomorphic to $m_2\A^1$, and the two positive integers
$m_1$ and $m_2$ satisfy the condition $\gcd(m_1,\,m_2)\,=\,1$. We call such a curve $C$ a \emph{cross}.
\end{definition}

The following two lemmas are proved in \cite{GGH}:

\begin{lemma}[{\cite[Lemma 4.3]{GGH}}]\label{2nd cohom. with compact support of a cross}
For a cross $C$, the compactly supported cohomology $$H^2_c(C,\, {\mathbb Z})$$ is isomorphic to ${\mathbb Z}\oplus{\mathbb Z}$.
\end{lemma}

\begin{lemma}[{\cite[Lemma 4.4]{GGH}}]\label{Lem: infinite b_2}
Let $f\,:\, X\, \longrightarrow\,B$ be a complex analytic $\C$-fibration from a smooth connected complex analytic
surface to a connected Riemann surface. Suppose at least one of the following two holds:
\begin{enumerate}
\item There is one singular fiber of $f$ that contains an infinite disjoint union of either crosses or analytic curves with reduced structure isomorphic to $\C$.

\item There are infinitely many singular fibers each consisting of curves with reduced structure isomorphic to $\C$ or a cross.
\end{enumerate}
Then $b_2(X)$ is infinite.
\end{lemma}

Now we prove a result in a set-up similar to that of Lemma \ref{Lem: infinite b_2}. The main ideas of the proof of the 
following lemma are not very different from those in \cite[Lemma 4.4, Proposition 4.5]{GGH}. But systematic changes are needed
in the arguments of \cite[Lemma 4.4, Proposition 4.5]{GGH}.

\begin{lemma}\label{Lem: infinite b_2 for C*-fibration having affine lines in fibers}
Let $f\,:\, X\, \longrightarrow\,B$ be a complex analytic $\C^*$-fibration, from a smooth connected complex
analytic surface to a connected Riemann surface, with no multiple fiber. Then
the following two statements hold:
\begin{enumerate}
\item If there is one fiber of $f$ containing infinitely many disjoint crosses or an infinite disjoint union of analytic
curves with reduced structure isomorphic to $\C$, then both $b_2(X)$ and $\rank(\pi_2(X))$ are infinite.

\item If $B$ is not isomorphic to $\PP^1$, and there are infinitely many fibers of $f$ each containing either at least two curves
with reduced structure isomorphic to $\C$ or a cross, then $\rank (\pi_2(X))$ is infinite. 
\end{enumerate}
\end{lemma}

\begin{proof}
Let $\pi\,:\,\wt{B}\,\, \longrightarrow\, \,B$ be a universal covering. Let
\begin{equation}\label{j2}
f'\,\,:\,\, X'\,\,:=\,\, \wt{B}\times_B X\,\, \longrightarrow\,\,\wt{B}
\end{equation}
be the pullback of the fibration $f$ via $\pi$. Like $f$, the analytic map $f'$ is a $\C^*$-fibration with no
multiple fibers. Applying Theorem \ref{Nori's Lemma} to
$f'$ in \eqref{j2} it is deduced that either $\pi_1(X')$ is a finite cyclic group or $\pi_1(X')$ is isomorphic to $\Z$. Let
$$
\pi'\,\, :\,\, \wt{X}\, \,\longrightarrow\,\, X'
$$
be the universal covering. Clearly, $\wt{X}$ is a universal cover of $X$ as well. Let
\begin{equation}\label{j8}
\wt{f}\,\,:=\,\, f'\circ \pi'\,:\, \wt{X}\,\, \longrightarrow\,\,\wt{B}.
\end{equation}
Then $\wt{f}$ is an $\wt{F}$-fibration, where
\begin{itemize}
\item $\wt{F}\,\cong\, \C^*$ when $\pi_1(X')$ is finite, and

\item $\wt{F}\,\cong\, \C$ when $\pi_1(X')\,\cong\, \Z$.
\end{itemize}
This is because $\wt{F}$ is a unramified covering of $\C^*$ whose degree coincides with the order of $\pi_1(X')$.

{\bf Case A.}\, \emph{Assume that $\pi_1(X')$ is infinite.}

As observed above, in this case $\wt{f}$ is an analytic $\C$-fibration. Note that if $f$ satisfies the assumption in (1) (respectively,
(2)) in the statement of this lemma, then $\wt{f}$ also the assumption in (1) (respectively,
(2)) in the statement of this lemma. Using Theorem \ref{Hurewicz Theorem}, $$\rank (\pi_2(X))\,=\,\rank(\pi_2(\wt{X}))\,=\,b_2(\wt{X}).$$
Since $b_2(\wt{X})$ is infinite by Lemma \ref{Lem: infinite b_2}, this implies that $\rank (\pi_2(X))$ is infinite.

{\bf Case 1.}\, \emph{Assume that there is a fiber of $f$ containing infinitely many disjoint union of either crosses or analytic curves with
reduced structure isomorphic to $\C$ (the assumption in {\rm(1)} in the statement of this lemma).}

Let $p\,\in\, B$ be a point such that the fiber $F\,:=\,f^{-1}(p)$ contains either infinitely many disjoint crosses or
an infinite disjoint union of analytic curves with reduced structure isomorphic to $\C$. Express $F$ as a disjoint union
$$F\,=\, \Gamma \sqcup \Delta ,$$ where 
\begin{enumerate}
\item[(a)] $\Gamma$ is a disjoint union of crosses $\{C_{j}\, \big\vert\, \,j \,\in\, J\}$, with $J$
being an index set ($J$ can also be empty), and

\item[(b)] $\Delta$ is a disjoint union of curves $\{A_{i}\,\big\vert\, i \,\in \,I\}$ such that each
$A_{i}$ is isomorphic to a multiple $\C$; here $I$ is an index set ($I$ can also be empty).
\end{enumerate}
It follows from the hypothesis that one of the above indexing sets $I$ and $J$ must be infinite. Let $D$ be a suitable small
open Euclidean real $2$-disc in $B$ around the point $p$ such that $$f|_{V\setminus F}\,\,:\,\, V\setminus F \,\, \longrightarrow
\,\,D\setminus \{p\}$$ is a trivial $\C^*$-bundle, where $V\,:=\,f^{-1}(D)$. Therefore, we have $V\setminus F\,\cong
\, (D\setminus \{p\})\times \C^*$. Consider the following part of the long exact sequence of cohomologies with compact support for
the pair $(V,\,F)$:
\begin{equation}\label{j6}
\begin{tikzcd}[row sep=small, column sep=small]
	{H^2_c(V)} && {H^2_c(F)} && {H^3_c(V,\,F)} \\
	& {\ker(\delta)} \\
	\arrow[two heads, from=1-1, to=2-2]
	\arrow[hook, from=2-2, to=1-3]
	\arrow[from=1-1, to=1-3]
	\arrow["\delta", from=1-3, to=1-5]
\end{tikzcd}
\end{equation}
Using duality, we have $H^3_c(V,\,F) \,\cong\, H_1(V\setminus F)\,\cong\, H_1(D\setminus \{p\})\oplus H_1(\C^*)
\,\cong\, \Z\oplus \Z$ and
\begin{equation}\label{j5}
H^2_c(F)\,\cong\, 
\left(\bigoplus\limits_{i \in I}H^2_c(A_{i})\right)\bigoplus\left(\bigoplus\limits_{j \in J}H^2_c(C_{j})\right).
\end{equation}
By duality, $H^2_c(A_{i})\,\cong\, H_0(A_{i})\,\cong\, H_0(\C)\,\cong\, \Z$. Lemma \ref{2nd cohom. with compact support of a cross}
implies that $H^2_c(C_{j})\,\cong\, \Z \oplus \Z$ for all $j \,\in \,J$. Since one of the indexing sets $I$ and $J$ is infinite, it
follows from \eqref{j5}
that $H^2_c(F)$ is a free abelian group of infinite rank. Therefore, $\ker(\delta)$
in \eqref{j6} is a free abelian group of infinite rank. So $H^2_c(V)$ is of infinite rank as 
$H^2_c(V)$ surjects onto $\ker(\delta)$. Hence using the duality again we conclude that $H_2(V)$ is of infinite rank, i.e.,
$b_2(V)$ is infinite.

Set $W\,:=\,X\setminus F$. Thus the pair of open subsets $\{V,\, W\}$ gives a Euclidean open covering of $X$. Consider
the following part of the Mayer-Vietoris sequence applied to the open covering $\{V, \,W\}$ of $X$:
\begin{equation}\label{j7}
\begin{tikzcd}[row sep=small, column sep=small]
&&& {\Ima(\mu)}\\
{H_2(V\cap W)} && {H_2(V)\oplus H_2(W)} && {H_2(X)} \\
& {\ker(\mu)} \\
\arrow[two heads, from=2-1, to=3-2]
\arrow[hook, from=3-2, to=2-3]
\arrow[from=2-1, to=2-3]
\arrow["\mu", from=2-3, to=2-5]
\arrow[two heads, from=2-3, to=1-4]
\arrow[hook, from=1-4, to=2-5]
\end{tikzcd}
\end{equation}
Now by K{\"u}nneth formula, $$H_2(V\cap W)\,=\,H_2(V\setminus F)\,=\,H_1(D\setminus \{p\})\otimes_{\Z}H_1(\C^*)
\,\cong\,\Z\otimes_{\Z}\Z\,\cong\, \Z.$$ Thus from \eqref{j7} we have $\rank(\ker(\mu))\,\leq\, 1$. Since $b_2(V)$ is infinite,
from \eqref{j7} it follows that $\rank(\Ima(\mu))$ is infinite and consequently $H_2(X)$ is of infinite rank by \eqref{j7}.

Now consider $\wt{f}$ in \eqref{j8}. Observe that $\wt{f}$ has one fiber containing infinitely many disjoint crosses or an infinite disjoint union of
analytic curves with reduced structure isomorphic to $\C$ since $f$ has one such fiber by our assumption in this case. Note that, $\wt{f}$ is a $\C^*$-fibration
(respectively, $\C$-fibration) whenever $\pi_1(X')$ is finite (respectively, infinite). 
\begin{itemize}
\item If $\wt{f}$ happens to be a $\C^*$-fibration then it satisfies the condition (1) in the statement of this lemma, then by what we did above (in this case) we can write once again a similar Mayer-Vietoris sequence as in \eqref{j7} applied to an open covering of $\wt{X}$ and it turns out finally that $b_2(\wt{X})$ is infinite. 

\item Next, if $\wt{f}$ happens to be a $\C$-fibration then it satisfies the condition (1) in the statement of Lemma \ref{Lem: infinite b_2} and therefore it follows that $b_2(\wt{X})$ is infinite.
\end{itemize}

So in either case, we prove that $b_2(\wt{X})$ is infinite. Using Theorem \ref{Hurewicz Theorem}, $$\rank(\pi_2(X))\,
=\,\rank(\pi_2(\wt{X}))\,=\,b_2(\wt{X}),$$ and hence $\rank(\pi_2(X))$ is infinite, because $b_2(\wt{X})$ is infinite.

{\bf Case 2.}\, \emph{Assume that $B$ is not isomorphic to $\PP^1$, and there are infinitely many fibers of $f$ each containing either
at least two curves with reduced structure isomorphic to $\C$ or a cross (the assumption in {\rm(2)} in the statement of this lemma).}

In view of Case A, we assume that $\pi_1(X')$ is finite. Therefore, $\wt{f}$ in \eqref{j8} is a $\C^*$-fibration. The map $\wt{f}$
satisfies the condition in statement (2) because $f$ does so. Let
\begin{equation}\label{j9}
\wt{S}\,:=\,\{q_1,\, q_2,\, \cdots\}\, \subset\, \wt{B}
\end{equation}
be the infinite subset of $\wt{B}$ such that each fiber $\wt{f}^{-1}(q_i)$ either contains a cross or contains at least two disjoint 
curves each isomorphic to an affine line. Since $B$ is not isomorphic to $\PP^1$, the uniformization theorem implies that $\wt{B}$ is 
homeomorphic to $\R^2$. So in the rest of the proof, we will take $\wt{B}$ to be $\R^2$.

We can find a suitable covering $\{U_i\,\,\big\vert\,\, i \,\in \,\N\}$
of $\R^2$ by \textit{closed} squares whose sides are parallel to the coordinate axes in $\R^2$ satisfying the following conditions:
\begin{enumerate}
\item[(a)] For $i \,\neq\, j$, the intersection $U_i\cap U_j$ is either empty, or one point, or one of the sides that is common to both $U_i$ 
and $U_j$.

\item[(b)] Any $q_i$ in \eqref{j9} is contained in the interior of some $U_j$. For any $k \,\in\, \N$, if $U_k\cap\wt{S} \,=\,\emptyset$, 
then the fibration $f'$ in \eqref{j2} is topologically trivial over $U_k$. Clearly, any $q_i$ is contained in the interior of at most one 
$U_j$.

\item[(c)] Any $U_i$ contains at most one $q_j$ in its interior.

\item[(d)] For every $k\,\in\, \N$, the intersection $(U_1 \cup U_2 \cup \cdots \cup U_k) \cap U_{k+1}$ is either a side of $U_{k+1}$ or a 
union of two sides of $U_{k+1}$, meeting at one of their end points. In particular, for every $k\,\in\, \N$, the intersection $(U_1 \cup U_2 
\cup \cdots \cup U_k) \cap U_{k+1}$ is a non-empty, compact and contractible subset of $\R^2$ such that $\wt{f}$ is a locally trivial 
$\C^*$-bundle over this intersection.
\end{enumerate}
Let $V_i\,:=\, \wt{f}^{-1}(U_i)$ for all $i\,\in\, \N$; thus $\{V_i\,\big\vert\, i \,\in\, \N\}$ is a covering of $\wt{X}$. Take
a positive integer $k$. Consider the following part of the Mayer-Vietoris sequence:
$$H_2((V_1 \cup V_2 \cup \cdots \cup V_k) \cap V_{k+1})\,\longrightarrow\,
H_2(V_1\cup\cdots\cup V_k) \oplus H_2(V_{k+1})
$$
$$
\,\longrightarrow\, H_2(V_1\cup\cdots\cup V_k\cup V_{k+1}).$$ From (d) above it follows that
$(V_1 \cup V_2 \cup \cdots \cup V_k) \cap V_{k+1}$ is actually a $\C^*$-bundle over
$(U_1 \cup U_2 \cup \cdots \cup U_k) \cap U_{k+1}$. Since the intersection $(U_1 \cup U_2 \cup \cdots \cup U_k) \cap U_{k+1}$ is
contractible, therefore $(V_1 \cup V_2 \cup \cdots \cup V_k) \cap V_{k+1}$ contracts to a fiber that is $\C^*$.
As $H_2(\C^*)\,=\,0$, it follows that $H_2((V_1 \cup V_2 \cup \cdots \cup V_k) \cap V_{k+1})$ is also trivial. Therefore, the
homomorphism $$H_2(V_1\cup\cdots\cup V_k) \oplus H_2(V_{k+1})\,\longrightarrow\, H_2(V_1\cup\cdots\cup V_k\cup V_{k+1})$$ is injective.

Since $\{U_i\,\big\vert\, i \,\in \,\N\}$ covers the entire $\R^2$, for a sufficiently large $k\,\geq\, 1$, the intersection
$(U_1 \cup U_2 \cup \cdots \cup U_k)\cap \wt{S}$ is non-empty. Again since $\wt{S}$ is infinite, it follows from (c) that
as $k$ increases the number of points in the above intersection eventually increases. As the Euler characteristic
of a cross or an affine line is 
positive, and $\wt{S}$ is an infinite set, using the (finite) additivity property of Euler characteristic we conclude that 
$e(V_1\cup\cdots\cup V_k)$ tends to infinity as $k$ tends to infinity. Since $V_1\cup\cdots\cup V_k$ is a real $4$-manifold which is 
non-compact, we have $b_4(V_1\cup\cdots\cup V_k)\,=\,0$. Therefore, from the fact that
$\lim_{k \to \infty}e(V_1\cup\cdots\cup V_k)\,=\,\infty$ it is deduced that 
$b_2(V_1\cup\cdots\cup V_k)$ tends to infinity as $k$ tends to infinity. Since homology groups commute with direct limits, it
follows that $b_2(\wt{X})$ is infinite. Therefore, Theorem \ref{Hurewicz Theorem} gives that $\rank(\pi_2(X))$ is infinite.
\end{proof}

\begin{remark}
If Case 2 in the proof of Lemma \ref{Lem: infinite b_2 for C*-fibration having 
affine lines in fibers} occurs, and the two additional assumptions 
that $X$ is a Stein manifold and $\pi_1(X)$ is finite hold, then $b_2(X)$ is infinite. To see 
this, note that if $b_2(X)$ is finite, and $X$ is a Stein surface, then 
$e(X)\,=\,1+b_2(X)$ is finite, and thus $$\rank 
(\pi_2(X))\,=\,b_2(\wt{X})\,=\,e(\wt{X})-1\,=\,|\pi_1(X)|\cdot e(X)-1$$ is also finite. But 
this contradicts the second statement of Lemma \ref{Lem: infinite b_2 for C*-fibration 
having affine lines in fibers}.
\end{remark}

\subsection{Case when $\Kbar\,=\,-\infty$}

\begin{theorem}\label{Thm: Kappa negative, pi_2 finite rank}
Let $X$ be a smooth complex affine surface with $\Kbar(X)\,=\,-\infty$ such that $\pi_1(X)$ is
infinite and $\rank(\pi_2(X))$ is finite. Then $X$ is an $\A^1$-bundle over a smooth algebraic
curve $B$ not isomorphic to $\A^1$ or $\PP^1$.
Consequently, $X$ is an Eilenberg--MacLane $K(\pi_1(X),\,1)$-space.
\end{theorem}

\begin{proof}
Given that $\Kbar(X)\,=\,-\infty$, a well-known result due to Fujita--Miyanishi--Sugie (cf. \cite{Fuj1}, \cite{Miy-Sug}, \cite{Sug})
says that $X$ admits an $\A^1$-fibration $$f\,:\, X\, \longrightarrow\,B$$ onto a smooth algebraic curve
$B$. It is also known that every fiber of $f$, if
taken with the reduced structure, is a finite disjoint union of smooth affine curves each isomorphic to $\A^1$. (cf. \cite[Chapter I, Lemma 4.4]{Miy1}). 

{\bf Claim.}\, \emph{If $B\,\cong\, \PP^1$, then $f$ has at least three multiple fibers.}

{\it Proof of Claim.} Assume that $B\,\cong\, \PP^1$. If $f$ has at most one multiple fiber, 
then Lemma \ref{Lem: Gang's Generalization} immediately proves that 
$X$ is simply connected. This contradicts the hypothesis that $\pi_1(X)$ is infinite.

Assume that $f$ has exactly two multiple fibers $F_1$ and $F_2$ over points $p_1$ and $p_2$
respectively, with $\mult(F_i)\,=\,m_i$ for $i\,=\,1,\,2$. Let
$$\gcd(m_1,\,m_2)\,=\,d.$$ Then applying Lemma \ref{Lem: Gang's Generalization} to $f$ we obtain the following exact sequence:
\begin{equation}
(1)\,=\,\pi_1(\A^1)\,\longrightarrow\, \pi_1(X)\,\longrightarrow\, \Z/(\gcd(m_1,\, m_2))
\,=\,\Z/d\Z,
\end{equation}
as $B\setminus \{p_1,\, p_2\}\,\cong\, \C^*$ has fundamental group $\Z$. This
contradicts the fact that $\pi_1(X)$ is infinite. 

This completes the proof of the claim.

The above Claim ensures that the usual ramified covering trick is applicable here (cf. Theorem \ref{Ramified covering trick}).
Thus using the ramified covering trick we get another $\A^1$-fibration
$$f'\,:\, X'\, \longrightarrow\,B'$$ such that $X'$ is a finite cover of $X$ and $f'$ does
not have any multiple fibers. From Theorem \ref{Nori's Lemma} it follows that $\pi_1(X')\,\cong\, \pi_1(B')$. Since $X'$ is a finite cover
of $X$ and $\pi_1(X)$ is infinite, we know that $\pi_1(X')$ is also infinite. Hence $B'$ is not isomorphic to $\A^1$ or $\PP^1$.
	
Assume that $f'$ is not an $\A^1$-bundle. Since $f'$ does not have any multiple fiber, there is
a point $p\,\in\, B'$ such that the fiber $F\,=\, (f')^{-1}(b)$ consists of at least two disjoint curves each of which, if taken with reduced structure, isomorphic
to $\A^1$. Now take a
universal covering map $$\pi\,:\, \wt{B}'\, \longrightarrow\,B',$$ and
consider the fiber product $\wt{X}' \,:=\, X'\times_{B'}\wt{B'}$. 
Then $\wt{X}'$ is a smooth complex analytic surface which is no longer algebraic, and $\wt{X}'$ is an infinite sheeted covering of 
$X'$. Using \cite[Remarks 3.13, 3.14]{GGH}, it follows that $\wt{X}'$ is simply connected. Therefore in the new complex analytic $\C$-fibration $$\wt{f}'\,:\, \wt{X}'\, \longrightarrow\,\wt{B}'$$ there are infinitely many fibers, over points in $\pi^{-1}(p) \, \subset\, \wt{B}'$, each of which contains a disjoint union of at least two smooth curves each homeomorphic to $\C$. Consequently, by Lemma \ref{Lem: infinite b_2} and Theorem \ref{Hurewicz Theorem} it follows that $$b_2(\wt{X}')\,=\,\dim 
H_2(\wt{X}';\,\Q)\,=\,\dim\pi_2(\wt{X}')\otimes_\Z\Q\,=\,\dim\pi_2(X)\otimes_\Z\Q$$ is infinite, which contradicts
the given condition that $\dim\pi_2(X)\otimes_\Z\Q\, <\, \infty$. Hence
$f'\,:\, X'\, \longrightarrow\,B'$ is an $\A^1$-bundle.

Since $f'\,:\, X'\, \longrightarrow\,B'$ is an $\A^1$-bundle, every fiber of $f$ is reduced and irreducible. Hence, $f\,:\, X\, 
\longrightarrow\,B$ is also an $\A^1$-bundle. We know that $B$ is not isomorphic to $\A^1$ or $\PP^1$
because $\pi_1(X)$ is infinite. Finally, using 
Theorem \ref{Result : Classification} it follows that $X$ is indeed an Eilenberg--MacLane $K(\pi,\,1)$-space.
\end{proof}

\subsection{Case when $\Kbar\,\in\,\{0,\,1\}$}
 
The following lemma will be used later.

\begin{lemma}\label{Lem: e=0 surface admitting C*-fibration over P1 is never affine}
Let $f\,:\, X\, \longrightarrow\,\PP^1$ be a complex analytic map from a smooth complex analytic surface $X$ to $\PP^1$ such that every fiber of $f$, if taken with reduced structure, is isomorphic to $\C^*$. Then $X$ is not a Stein surface.
\end{lemma}

\begin{proof}
To prove by contradiction, assume that $X$ is a Stein space. 

{\bf Case 1.}\, \emph{Assume that $f$ is a $\C^*$-bundle.}

Take a point $p\,\in\, \PP^1$. Using the local triviality of $f$ around $p$ fix a
sufficiently small Euclidean open disc $p\, \in\, \Delta\, \subset\, \PP^1$
such that $$V\,:=\,f^{-1}(\Delta)\,\cong\, \Delta \times \C^* .$$ Denote
$F\,:=\, f^{-1}(p)$ which is isomorphic to $\C^*$, and set
$U\,:=\,X\setminus F$. Thus $\{U,\, V\}$ is an open cover of $X$ for the analytic topology.
Consider the following part of the Mayer-Vietoris sequence for integral homology for $X\,=\, U \cup V$:
\begin{equation}\label{e10}
H_3(U)\oplus H_3(V)\,\longrightarrow\, H_3(X)
\,\longrightarrow\, H_2(V\setminus F)\,\longrightarrow\, H_2(U)\oplus H_2(V).
\end{equation}
By K{\"u}nneth formula,
\begin{equation}\label{e12}
H_2(V;\,\Z)\,=\,0\,=\,H_3(V;\,\Z).
\end{equation}
Since $$f\big\vert_U\,:\, U\, \longrightarrow\,\PP^1\setminus \{p\}\,\cong\, \C$$
is a $\C^*$-bundle, and $\C$ is contractible, we have the following Serre's long exact sequence of integral homologies:
\begin{equation}\label{e11}
\cdots \,\longrightarrow\, H_i(\C^*)\,\longrightarrow\, H_i(U)
\,\longrightarrow\, H_i(\C)\,\longrightarrow\, \cdots
\end{equation}
for all $i\,\geq\, 0$. From \eqref{e11} it follows that $H_2(U;\,\Z)\,=\,0\,=\,H_3(U;\,\Z)$. Using this and \eqref{e12} in \eqref{e10} we 
conclude that $H_3(X;\,\Z)\,\cong\, H_2(V\setminus F;\,\Z)$. Since $$H_2(V\setminus F;\,\Z)\,\cong \, H_2((\Delta\setminus \{p\})\times\C^*;\,\Z)\,\cong\, 
\Z ,$$ this implies that $H_3(X;\,\Z)\,\cong\, \Z$. But we have $H_3(X;\,\Z)\,= \, 0$, for $X$ being a complex Stein space, by Theorem \ref{Narasimhan's Theorem}.

In view of the above contradiction, we now consider the complementary case.

{\bf Case 2.}\, \emph{Assume that $f$ has at least one singular fiber.}

Since $f\,:\, X\, \longrightarrow\, \PP^1$ is an algebraic map, it can have at most finitely 
many singular fibers. Also, by the hypothesis, every singular fiber is scheme-theoretically 
isomorphic to $m\C^*$ for some $m\,\geq\, 1$. Let $F_1,\, F_2,\, \cdots,\, F_r$ be all the singular 
fibers, and let $f(F_i)\,:=\, p_i \,\in\, \PP^1$, $1\,\leq\, i\,\leq\, r$; so $r$ is a positive integer. 
Let $S\,:=\,\{p_1,\, \cdots,\, p_r\}$ and $D\,:=\,\bigsqcup_{i=1}^{r}F_i$. Since $X$ is a smooth Stein 
surface, so is $X\setminus D$.

Now we consider the following part of the long exact sequence of cohomologies with compact support (with integer coefficients) for the 
pair $(X,\, D)$:
$$H^1_c(X,\, D)\,\longrightarrow\, H^1_c(X)\,\longrightarrow\, H^1_c(D) \,\longrightarrow\, H^2_c(X,\, D).$$
Using duality and Theorem \ref{Narasimhan's Theorem} it follows that $H^1_c(X,\,D;\,\Z)\,\cong\,
H_3(X\setminus D;\,\Z)\,=\,0$. Therefore,
\begin{equation}\label{e12n}
H^1_c(X;\,\Z)\,\cong\,
\ker\left(H^1_c(D;\,\Z) \,\longrightarrow\, H^2_c(X,\,D;\,\Z)\right).
\end{equation}
Again,
$$H^1_c(D;\,\Z)\,=\,\bigoplus_{i=1}^{r}H^1_c(F_i;\,\Z)\,\cong\, \bigoplus_{i=1}^{r}H_1(F_i;\,\Z)\,=\,\Z^r,$$ since
$(F_i)_{\text{red}}\,\cong\, \C^*$ for all $i\,\leq\, i \,\leq\, r$ and $H_1(\C^*;\,\Z)\,=\,\Z$. Now it is evident that
$$f\big\vert_{X\setminus D}\,:\, X\setminus D\, \longrightarrow\,\PP^1\setminus S$$ is a $\C^*$-bundle. Therefore $$e(X-D)\,=\,
e(\C^*)\cdot e(\PP^1\setminus S)\,=\,0.$$ Also note that since $X\setminus D$ is a Stein surface, $e(X-D)\,=\,1-b_1(X\setminus D)+b_2(X\setminus D)$,
whence $b_2(X\setminus D)\,=\,b_1(X\setminus D)-1$. Also by Theorem \ref{Nori's Lemma} followed by taking abelianization, we get an exact sequence of integral homologies
$$H_1(\C^*)\,\xrightarrow{\,\,\,\phi\,\,}\, H_1(X\setminus D)\,\longrightarrow\, H_1(\PP^1\setminus S)\,\longrightarrow\, 
0.$$ Here $H_1(\C^*;\,\Z)\,\cong\, \Z$, and $H_1(\PP^1\setminus S;\,\Z)\,\cong \,\Z^{r-1}$. Clearly, the following are true:
\begin{itemize}
\item $b_1(X\setminus D)
\,=\,b_1(\PP^1\setminus S)+\rank(\Ima(\phi))$; and 
\item $\rank(\Ima(\phi))\,\leq\, b_1(\C^*)\,=\,1$.
\end{itemize}
These together imply that
$b_1(X\setminus D)\,\leq\, r$, which in turn implies that $$b_2(X\setminus D)\,=\,b_1(X\setminus D)-1\,\leq\, r-1.$$
By duality, $H^2_c(X,\,D;\,\Z)\,\cong\, H_2(X\setminus D;\,\Z)$ and $H^1_c(X;\,\Z)\,\cong\, H_3(X;\,\Z)$. Thus using \eqref{e12n} it follows that $H_3(X;\,\Z)\,\cong\,
\ker\left(H^1_c(D;\,\Z) \,\longrightarrow\, H^2_c(X,\,D;\,\Z)\right)$ and $$\rank H^2_c(X,\,D;\,\Z)\,=\,b_2(X\setminus D)\,\leq\, r-1 \,< \,r\,=\,\rank H^1_c(D;\,\Z).$$
Hence we conclude that
\begin{equation}\label{tt}
\rank H_3(X;\,\Z) \,\geq \,1.
\end{equation}
But $X$ being a Stein surface, $H_3(X;\,\Z)\,=\,0$ by Theorem \ref{Narasimhan's Theorem},
and this contradicts \eqref{tt}. This completes the proof.
\end{proof}

The next result generalizes \cite[Proposition 4.5]{GGH}.

\begin{theorem}\label{Surfaces admitting C^*-fibration with pi_2 of finite rank}
Let $X$ be a smooth complex affine surface, with $$|\pi_1(X)|\,=\, \infty\ \ \, and\ \ \,
\dim \pi_2(X)\otimes_\Z \Q\, <\, \infty,$$
such that $X$ admits a $\C^*$-fibration $f\,:\, X\, \longrightarrow\, B$ onto a smooth algebraic curve. Then the following assertions hold:
\begin{enumerate}
\item Every fiber of $f$ , if taken with reduced structure, is isomorphic to $\C^*$.

\item $e(X)\,=\,0$.

\item $X$ is an Eilenberg--MacLane $K(\pi,\,1)$-space.
\end{enumerate}
\end{theorem}

\begin{proof}
Let $\wt{X}$ be a universal cover of $X$. Then $\wt{X}$ is a smooth complex Stein surface. From the
given condition on $X$ it is clear that $\dim\pi_2(\wt{X})\otimes_\Z\Q\,=\,b_2(\wt{X})$ is finite.

{\it Proof of (1).}
Consider $f\,:\, X\, \longrightarrow\,B$ with $B\,\cong\, \PP^1$. We will first prove the following useful lemma and then continue the proof again.

\begin{lemma}\label{Lem: Lemma (inside)}
$f$ must contain at least three multiple fibers whenever $B\,\cong\, \PP^1$.
\end{lemma}

{\it Proof of Lemma.}
We prove this by splitting into two claims.

{\bf Claim A.}\, \emph{$f$ must have at least two multiple fibers.}

{\it Proof of Claim A.} It follows from Case 1 of Lemma \ref{Lem: e=0 surface admitting C*-fibration over P1 
is never affine} that $f$ is not a $\C^*$-bundle.

Suppose that either $f$ has no multiple fiber at all or it has exactly one multiple fiber $F_0$ over a point $p_0 \,\in\, \PP^1$ with $\mult(F_0)\,=\,m$. Then applying Lemma \ref{Lem: Gang's Generalization} to $f$ we have an exact sequence 
$$\Z\,=\,\pi_1(\C^*) \,\longrightarrow\, \pi_1(X) \,\longrightarrow\, (1).$$ Since $\pi_1(X)$ is infinite, the above surjective 
homomorphism $\Z\,=\,\pi_1(\C^*) \,\longrightarrow\, \pi_1(X)$ is in fact an isomorphism. Let $F$ denote a general fiber of $f$, which 
is isomorphic to $\C^*$. Since the inclusion map $i \,:\, F \,\hookrightarrow\, X$ induces an isomorphism $$i_\ast\,:\, \pi_1(F) 
\,\longrightarrow\, \pi_1(X),$$ we conclude that $\wt{F}\,:=\,F\times_{X}\wt{X}$ is a universal cover of $F$, and this covering is 
infinite sheeted, because $\pi_1(F)\,\cong\, \pi_1(\C^*)\,\cong\, \Z$. Since $F \,\cong\, \C^*$, this proves that $\wt{F}$ is topologically isomorphic to $\C$ and the induced map $\wt{f}\,:\, \wt{X}\, \longrightarrow\,B$ is an analytic $\C$-fibration.

If the fiber of $f$ over a point $b\,\in\, B$ contains an irreducible component whose
reduced structure is isomorphic to $\A^1$, then the fiber 
$\wt{f}^{-1}(b)$ contains either infinitely many copies of crosses or an infinite number of disjoint curves each isomorphic 
to $\A^1$. Note that this uses that both $\A^1$ and a cross are topologically contractible. Now Lemma \ref{Lem: infinite b_2} says that $b_2(\wt{X})$ is infinite. However, this contradicts the hypothesis 
that $\pi_2(X)$ is of finite rank. Therefore, every singular fiber of $f$ is a multiple $\C^*$.

Since every singular fiber of $f$ is a multiple $\C^*$, and $X$ is a smooth variety, 
Lemma \ref{Lem: e=0 surface admitting C*-fibration over P1 is never affine} says that $X$ is not a Stein surface. But this
contradicts the fact that $X$ is affine. So $f$ can't have exactly one multiple fiber.

{\bf Claim B.}\, \emph{$f$ can not have exactly two multiple fibers.}

{\it Proof of Claim B.}\, Let $f\,:\, X\, \longrightarrow\,\PP^1$ have two multiple fibers $F_1$ and $F_2$ over the points $p_1, \,p_2
\,\in\, \PP^1$ respectively with $\mult(F_i)\,=\,m_i$ for $i\,=\,1,\,2$. We can assume that $p_1\,=\,0$ and $p_2\,=\,\infty$ in $\PP^1$. Let $\gcd(m_1,\,m_2)
\,=\,d$. Consider the analytic map $\Phi\,:\, \PP^1\, \longrightarrow\,\PP^1$ given by $\Phi(\infty)\,=\,\infty$ and $\Phi(z)\,=\,z^d$
for all $z\,\in \,\PP^1\,\setminus\,\{\infty\}\,=\,\C$. Then $g$ is a finite map of degree $d$. Let $X_1$ be the normalized fiber product with respect
to the maps $f$ and $\Phi$ over $\PP^1$. Then the induced map $X_1 \,\longrightarrow\, X$ is a covering map. Clearly $X_1$ is a smooth affine surface
such that $\pi_1(X_1)$ is infinite and $\rank (\pi_2(X_1))\,=\,\rank(\pi_2(X))\,<\,\infty$. Also $X_1$ admits a $\C^*$-fibration over $\PP^1$.

Now either the induced $\C^*$-fibration from $X_1$ to $\PP^1$ has at most one multiple fiber or if there are two multiple fibers then their multiplicities are coprime. But the first cannot happen by Claim A above. So we can assume, without loss of generality that $d\,=\,1$.

Now by Lemma \ref{Lem: Gang's Generalization} applied to $f$, we have an 
exact sequence $$\Z\,=\,\pi_1(\C^*) \,\longrightarrow\, \pi_1(X) \,\longrightarrow\, \Z/(\gcd(m_1,m_2))\,=\,(1).$$ Since $\pi_1(X)$ is 
infinite, the above surjective homomorphism $\Z\,=\,\pi_1(\C^*) \,\longrightarrow\, \pi_1(X)$ is in fact an isomorphism. Once again, 
using the same argument as in Claim A above, we can easily see that all singular fibers of $f$ are copies of multiple $\C^*$'s. But again this cannot happen because $X$ is a smooth affine surface and thus it would contradict Lemma \ref{Lem: e=0 surface admitting C*-fibration over P1 is 
never affine}. This completes the proof of this claim.

Hence the proof of the lemma follows.

{\it Back to the proof of Theorem \ref{Surfaces admitting C^*-fibration with pi_2 of finite rank}.} --- Lemma \ref{Lem: Lemma (inside)} implies that Theorem \ref{Ramified covering trick} is applicable to $f\,:\, X\, 
\longrightarrow\,B$ in order to eliminate the multiplicities of the fibers. Let $$\varphi\,:\, C\, \longrightarrow\, B$$ be a suitable 
finite ramified cover such that the induced $\C^*$-fibration $$\ol{f}\,:\,\ol{X \times _{B}C} \, \longrightarrow\, C$$ has no multiple 
fiber, where $\ol{X\times_B C}$ is the normalization of $X\times_B C$. The natural morphism $\ol{X \times _{B}C}\, \longrightarrow\,X$ 
is finite {\'e}tale and hence $\ol{X \times _{B}C}$ also has infinite fundamental group and $\pi_2(\ol{X \times_{B}C})\cong \pi_2(X)$. 
From the description of any singular fiber of a $\C^*$-fibration on a smooth affine surface, it follows that the reduced structure of 
every fiber of $\ol{f}$ is isomorphic to $\C^*$ if and only if the reduced structure of every fiber of $f$ is isomorphic to $\C^*$. 
Therefore, we can assume, without loss of generality, that $f\,:\, X\, \longrightarrow\,B$ has no multiple fiber.

Take a universal covering $\pi\,:\, \wt{B}\, \longrightarrow\,B$ of $B$. Since $f_\ast:\pi_1(X) \,\longrightarrow\,\pi_1(B)$ is surjective, the natural
map $$X'\,:=\,X\times_B \wt{B} \, \longrightarrow \,X$$ is a connected covering. Thus, whenever $B$ is simply connected,
$X'\,=\, X$, which is an algebraic variety. Note that $X'$ is not an algebraic variety in general; it is only a complex
manifold of complex dimension $2$. Since $f$ is a $\C^*$-fibration, the induced complex analytic map $f'\,:\, X'\, \longrightarrow\,\wt{B}$ is a
topological $\C^*$-fibration over $\wt{B}$. Since $f'$ has no multiple fiber, by using \cite[Remarks 3.13, 3.14]{GGH}, we have the following exact sequence:
$$\pi_1(\C^*) \,\longrightarrow\, \pi_1(X') \,\longrightarrow\, \pi_1(\wt{B})\,\longrightarrow\, (1).$$
Since $\wt{B}$ is simply connected, the above homomorphism $\pi_1(\C^*) \,\longrightarrow\, \pi_1(X')$ is surjective. As
$\pi_1(\C^*)\,\cong\, \Z$, it follows that $\pi_1(X')$ can either be trivial, or $\Z/n\Z$ for some $n\,>\,1$, or isomorphic to $\Z$.
	
Suppose, if possible, that a singular fiber $F_p$ of $f$ over $p\in B$ has an irreducible component isomorphic to $\A^1$ with reduced 
structure. Let $S$ be the finite set of all such points in $B$. From the assumption that $f\,:\, X\, \longrightarrow\, B$ has no multiple 
fiber, and the description of singular fibers of a $\C^*$-fibration, it follows that $S$ is precisely the set of points where $f$ is 
not locally trivial.

{\bf Claim 1.}\, \emph{$X'$ cannot be simply connected.}\mbox{}

{\it Proof of Claim 1.} Suppose, if possible, $X'$ is simply connected. Then $X'$ is a universal cover of $X$, and thus $\wt{X}$ and $X'$ are 
homeomorphic. Suppose that $B$ is simply connected. As we observed before, in this case, $X'\,=\, X$ by the construction. Hence $X$ is 
simply connected. However, this is a contradiction to the hypothesis that the fundamental group of $X$ is infinite.

Therefore $\pi_1(B)$ is infinite, i.e., $\pi\,:\, \wt{B}\, \longrightarrow\,B$ is an infinite sheeted covering. Let 
$$\pi^{-1}(S)\,:=\,\wt{S}\,=\,\{q_1,\,q_2,\, \cdots\}.$$ So $\wt{S}$ is a countably infinite discrete subset of $\wt{B}$ and this is 
precisely the set of points where $\wt{f}$ is not locally trivial. So $f'\,:\, X'\, \longrightarrow\,\wt{B}$ has infinitely many 
singular fibers, and they are all non-multiples. Again every fiber of $\wt{f}\,:\, \wt{X}\, \longrightarrow\, \wt{B}$ over any point 
$q\,\in\, \pi^{-1}(S)$ either contains a curve isomorphic to $\A^1$ with reduced structure or contains a cross. So using the second 
statement of Lemma \ref{Lem: infinite b_2 for C*-fibration having affine lines in fibers} it follows that $\pi_2(X)\,=\,\pi_2(X')$ is of 
infinite rank; but it contradicts the hypothesis. This proves Claim 1. 

{\bf Claim 2.}\, \emph{$\pi_1(X')$ cannot be isomorphic to $\Z/n\Z\,$ for any $n\,>\,1$.}\mbox{}

{\it Proof of Claim 2.} If possible let $\pi_1(X')\,\cong\,\Z/n\Z\,$ for some $n\,>\,1$. Thus $\wt{X}$ is an $n$-fold covering of $X'$, and since the later is a $\C^*$-fibration over $\wt{B}$, it follows that 
$\wt{f}\,:\, \wt{X}\, \longrightarrow\,\wt{B}$ is also a $\C^*$-fibration. Again the same analysis as in Claim 1 proves Claim 2.

{\bf Claim 3.}\, \emph{$\pi_1(X')$ cannot be isomorphic to $\Z$.}\mbox{}

{\it Proof of Claim 3.} Let $\pi_1(X')$ be isomorphic to $\Z$. Let $\wt{X'}$ be a universal cover of $X'$. Therefore $\wt{X'}$ and
$\wt{X}$ are isomorphic. Let $F'$ be a general fiber of $f'\,:\, X'\, \longrightarrow\,\wt{B}$, so $F' \,\cong\, \C^*$. In the case
under consideration, we know that the inclusion map $i \,:\, F' \,\hookrightarrow\, X'$ induces an isomorphism $i_\ast
\,:\, \pi_1(F') \,\longrightarrow\,\pi_1(X')$. Therefore, the fiber product $\wt{F}\,:=\,F'\times_{X'}\wt{X'}$ is a universal cover of
$F'$, and this covering is infinite sheeted, because $\pi_1(F')\,\cong\, \pi_1(\C^*)\,\cong \,\Z$. Since $F' \,\cong \,\C^*$, this
proves that $\wt{F}$ is topologically isomorphic to $\C$, and the induced map
$\wt{f}\,:\, \wt{X}\, \longrightarrow\, \wt{B}$ is an analytic $\C$-fibration. 

Let $\wt{S}\,:=\,\pi^{-1}(S)$. If $B$ is simply connected then $X'\,=\, X$
and $\wt{S}\,=\,S$. If $\pi_1(B)$ is infinite, then $\wt{S}$ is a countably infinite discrete subset of $\wt{B}$. Let $\wt{F}_z$ denote the 
fiber $\wt{f}^{-1}(z)$ of $\wt{f}$ over any point $z\,\in\, \wt{S}$.

We know that $\wt{F}_{q}$ is a non-empty disjoint union $\wt{\Gamma}_q \sqcup \wt{\Delta}_q$, where 
\begin{enumerate}
\item[(a)] $\wt{\Gamma}_q$ is either empty or a countably infinite disjoint union of crosses $\{C_{q,j}\,\,\big\vert\,\, j \,\in \,\N\}$, 

\item[(b)] $\wt{\Delta}_q$ is either empty or a countably infinite disjoint union of curves $\{A_{q,j}\,\,\big\vert\,\,j\, \in\, \N\}$
such that reduction of each $A_{q,j}$ is isomorphic to $\A^1$.
\end{enumerate}

Thus it follows from Lemma \ref{Lem: infinite b_2} that $b_2(\wt{X})$ is infinite. But this gives rise to a similar contradiction 
to the hypothesis that $\dim \pi_2(X)\otimes_{\Z}\Q\, <\, \infty$. Therefore Claim 3 follows.

Since all the possibilities of $\pi_1(X')$ have been ruled out under the assumption that there is a singular fiber of $f$ over that contains an irreducible component whose reduction is isomorphic to $\A^1$, we conclude that $f\,:\, X\, \longrightarrow\, B$ has no singular fiber which contains an irreducible component whose reduction is isomorphic to $\A^1$. Now the proof of this first part is completed using the description of singular fibers of a $\C^*$-fibration on a smooth affine surface. 

{\it Proof of (2)}. Suzuki's Formula (cf. \cite[Proposition 2]{Suz}) applied to the map $f$ yields $$e(X)\,=\,e(B)\cdot e(\C^*) \,+ 
\,\sum\limits_{i=1}^{r}\left(e(F_i)\,-\,e(\C^*)\right),$$ where $F_i$'s for $1\,\leq\,i\,\leq r$ denote the singular fibers of $f$.
Since $e(\C^*)\,=\,0$. it turns out that $e(X)\,=\,\sum_{i=1}^{r}e(F_i)$. By what we have proved above, it follows that
$$e(F_i)\ =\ e((F_i)_{\rm red})\ =\ e(\C^*)\,=\,0$$ for all $1\,\leq\,i\,\leq r$. Hence $e(X)\,=\,0$.

\begin{remark}\label{Rem: e=0 implies all fibers C* with reduced structure}
    This is a side remark related to the converse of the statement in (2). With the same notations as above, assume that $e(X)\,=\,0$. If the $\C^*$-fibration $f$ has a singular fiber having an irreducible component whose reduction is isomorphic to $\A^1$, the Euler characteristic of such fiber is positive and therefore $e(X)$ becomes positive too using Suzuki's formula, stated above. Thus $e(X)$ being zero would force $f$ to contain singular fibers only multiple $\C^*$'s from the very description of a singular fiber of any algebraic $\C^*$-fibration on any smooth affine surface (cf. \cite[Lemma 2.9]{Miy-Sug1}). This proves that the above two assertions (1) and (2) are in fact equivalent.
\end{remark}
{\it Proof of (3)}. Now we will prove why such surface becomes a $K(\pi,\,1)$-space. Since $X$ is affine, by Lemma \ref{Lem: e=0 surface admitting C*-fibration over P1 is never affine} it follows that $B\,\not\cong\,\PP^1$. Therefore by using the usual ramified covering trick, it turns out that there is a finite {\'e}tale morphism from an algebraic (which is of course affine) surface, say $\bar{X}$, onto $X$ such that $\bar{X}$ is a $\C^*$-bundle over an algebraic curve, say $\bar{B}$. Since $B\,\not\cong\,\PP^1$, thus $\bar{B}\,\not\cong\,\PP^1$ too as $\bar{B}$ dominates $B$ under a finite surjective morphism. Therefore both $\C^*$ and $\bar{B}$ are Eilenberg-MacLane $K(G,\,1)$-spaces, whence $\bar{X}$ using \cite[Lemma 3.5]{GGH} and hence $X$ is an Eilenberg-MacLane $K(\pi,\,1)$-space.

This completes the proof.
\end{proof}

\begin{example}[{\bf Examples of smooth affine surfaces with $\pi_2$ of infinite rank}]\mbox{}

Let $f(X,\,Y) \,\in\, \C[X,\,Y]$ be an irreducible polynomial such that a general fiber of the polynomial map $f\,:\, \C^2\, 
\longrightarrow\, \C$ is isomorphic to $\C^*$. Consider the irreducible curve $H$ in $\C^2$ defined 
by the equation $f\,=\,0$. Let $V\,:=\,\C^2\setminus H$. Then $V$ is a complex smooth affine surface. Let $f\,:\, V\, 
\longrightarrow\, \C$ be the morphism obtained by restricting to $V\,\subset\, \C^2$.

Consider the following part of the long exact sequence of integral cohomologies with compact support for the pair $(\C^2,\,H)$:
\begin{equation}\label{et}
H^2_c(\C^2;\,\Z)\,\longrightarrow\, H^2_c(H;\,\Z)\,\longrightarrow\, H^3_c(\C^2,\,H;\,\Z)\,\longrightarrow \,H^3_c(\C^2;\,\Z).
\end{equation}
Using duality, $H^2_c(\C^2;\,\Z)\,\cong\, H_2(\C^2;\,\Z)\,=\,0$, and similarly $H^3_c(\C^2;\,\Z)\,\cong\, H_1(\C^2;
\,\Z)\,=\,0$. So \eqref{et} gives that $H^3_c(\C^2,\,H;\,\Z)\,\cong\, H^2_c(H;\,\Z)$. Again by duality it follows that
$$H_1(V;\,\Z)\,\cong\, H^3_c(\C^2,\,H;\,\Z)\,\cong\, H^2_c(H;\,\Z)\,\cong\, H_0(H;\,\Z)\,\cong\, \Z.$$ 
Since $H_1(V;\,\Z)$ is infinite, $\pi_1(V)$ is also infinite. 
Therefore, it follows from Theorem \ref{Surfaces admitting C^*-fibration with pi_2 of finite rank} that $\pi_2(V)$ is of infinite rank.
\end{example}

Before moving to our next result we need some basic definitions.

In \cite[\S~2]{Koj}, Hideo Kojima has proved that every smooth affine surface with $\Kbar\,=\,0$ is rational (cf. \cite[Theorem 1.6]{Koj}). 
In the same paper, Kojima defined \emph{strongly minimal model} $(V,\,D)$ of a smooth open rational surface $X$ with $\Kbar(X)\,=\,0$, 
where $V$ is a smooth projective surface and $D\, \subset\, V$ is a simple normal crossing connected $\Q$-divisor.

\begin{definition}[{cf. \cite[Definition 2.6]{Koj}}]
Let $X$ be a smooth affine surface with $\Kbar(X)\,=\,0$. Then $X$ is called \emph{strongly minimal} if there is a strongly
minimal model $(V,\,D)$ of $X$ such that $X\,=\,V\setminus D$.
\end{definition}

\begin{lemma} [{\cite[Lemma 2.3]{Koj}}]\label{Strongly minimal is obtained by removing A^1's}
Let $S$ be a nonsingular open rational surface with $\Kbar(S)\, =\, 0$ and with a connected boundary at infinity. Let $(V,\, D)$ be a strongly minimal model of $S$. Then
the following statements hold:
\begin{enumerate}
\item $D$ is connected.

\item Suppose further that $S$ is affine. Then $V \setminus D$ is an affine open subset of $S$, and
$S \setminus (V \setminus D)$ is either an empty set or it is a disjoint union of curves isomorphic to $\A^1$. 
\end{enumerate}
\end{lemma}

Let $S$ be a smooth affine surface with $\overline{\kappa}(S)\,=\,0$. Then $S'\,=\,V\setminus D$ is a strongly minimal surface 
corresponding to a strongly minimal model $(V,\,D)$ of $S$ such that $S'$ is embedded in $S$ as an affine open subspace. It can be 
seen that $\Kbar(S)\,=\,\Kbar(S')\,=\,0.$

In \cite[\S~3--5]{Koj}, the author has classified the strongly minimal smooth affine surfaces $S'$ with $\Kbar(S')\,=\,0$ and having 
connected boundaries at infinity (see, \cite{Koj} for necessary definitions).

\begin{theorem}\label{Thm: Kappa 0, pi_2 finite rank}
Let $X$ be a smooth complex affine surface with $\Kbar(X)\,=\,0$ having infinite fundamental group and
$\dim \pi_2(X)\otimes_{\Z}\Q\,<\, \infty$. Then $X$ is either isomorphic to $\C^\ast\times \C^\ast$ or it is
Fujita's $H[-1,\,0,\,-1]$. Consequently, $X$ is an Eilenberg--MacLane $K(\pi,\,1)$-space.
\end{theorem}

\begin{proof}
Let $X_0$ be the strongly minimal model of $X$ and $i\,:\,X_0\,\hookrightarrow \,X$ the open immersion. Since $i_\ast\,:\,\pi_1(X_0) 
\,\longrightarrow\,\pi_1(X)$ is surjective, and $\pi_1(X)$ is infinite, it follows
that $\pi_1(X_0)$ is also infinite. Also $\Kbar(X_0)\,=\,\Kbar(X)\,=\,0$. 
So from the list of Kojima (cf. \cite{Koj}) it is evident that $X_0$ is isomorphic to one of $O(4,\,1)$, $O(2,\,2)$, $O(1,\,1,\,1)$, 
$H[-1,\,0,\,-1]$ and $H[0,\,0]$. Note that all the above five strongly minimal surfaces admit a $\C^\ast$-fibration $f_0\,:\, X_0\, 
\longrightarrow\, B_0$ onto a smooth algebraic rational curve $B_0$, which extends as a $\C^\ast$-fibration to $X$ onto a larger base 
curve $B$, i.e., there is a $\C^\ast$-fibration $f\,:\, X\, \longrightarrow\,B$ onto a smooth rational algebraic curve such that $B_0$ 
embeds in $B$ as a Zariski open subset and $f\big\vert_{X_0}\,=\,f_0$; this is observed in \cite{GGH}. Since $\pi_2(X)$ is of finite rank and 
$\pi_1(X)$ is infinite, Theorem \ref{Surfaces admitting C^*-fibration with pi_2 of finite rank} gives that $e(X)\,=\,0$.
Using Kojima's lemma (cf. Lemma \ref{Strongly minimal is obtained by removing A^1's}) we know that $e(X_0) \,\leq\, e(X)\,=\,0$ and,
moreover, the equality holds if and only if $i\,:\,X_0\,\hookrightarrow \,X$ is an isomorphism. In our case, the
five possibilities of $X_0$ as described above have non-negative Euler characteristics. Hence
$e(X_0)\,=\,e(X)\,=\,0$, i.e., either $X$ is isomorphic to $O(1,\,1,\,1)$, which is 
further isomorphic to $\C^\ast \times \C^\ast$, or $X$ is isomorphic to the surface $H[-1,\,0,\,-1]$.
 
Clearly $\C^*\times \C^*$ is an Eilenberg--MacLane $K(\Z^2,\,1)$-surface. It is proved in \cite[Proposition 5.8]{GGH} that the 
surface $H[-1,\,0,\,-1]$ is also an Eilenberg--MacLane $K(\pi,\,1)$-space with $\pi\,=\,\pi_1(H[-1,0,-1])$
being a non-abelian group. This completes 
the proof.
\end{proof}

\begin{prop}\label{Cor: Kappa 0 - strongly minimal model for an E-type surface}
Let $X$ be a smooth complex affine surface such that $X$ satisfies the finite homotopy rank-sum property and $\Kbar(X)\,=\,0$. Then the strongly minimal model for $X$ is either $\C^\ast\times \C^\ast$ or Fujita's $H[-1,\,0,\,-1]$. 
\end{prop}

\begin{proof}
The given condition implies that $\pi_2(X)$ is of finite rank. 

{\bf Case 1.}\, \emph{$\pi_1(X)$ is infinite.}

Theorem \ref{Thm: Kappa 0, pi_2 finite rank} says that $X$ is isomorphic to either $\C^*\times \C^*$ or Fujita's
$H[-1,\,0,\,-1]$. Therefore its strongly minimal model is itself.

{\bf Case 2.}\, \emph{$\pi_1(X)$ is finite.}

Since $X$ satisfies the finite homotopy rank-sum property, Corollary \ref{Cor:Finite Fundamental Group of Elliptic Type Affine Surface} says that $\pi_1(X)\,=\,\Z/2\Z$ with
$e(X)\,=\,1$. Let $X_0$ be the strongly minimal model of $X$. Recall that, $e(X_0)\,\leq\, e(X)$, and the equality holds if and only
if the usual open immersion $i\,:\, X_0 \, \hookrightarrow \, X$ is an isomorphism (see the proof of Theorem
\ref{Thm: Kappa 0, pi_2 finite rank}). Therefore,
\begin{enumerate}
\item[(i)] either $X\,=\,X_0$,

\item[(ii)] or $e(X_0)\,=\,0$.
\end{enumerate}

We will show that (i) does not occur. If $X\,=\, X_0$, then $X_0$ is a strongly minimal model with $e(X_0)\,=\,e(X)\,=\,1$ and
$\pi_1(X_0)\,=\,\Z/2\Z$. But from Kojima's list (cf. \cite[Table 1]{Koj}) it is evident that there does not exist such a surface.

Thus, $e(X_0)\,=\,0$. Once again from Kojima's list (cf. \cite[Table 1]{Koj}) it follows that $X_0$ is isomorphic to either
$\C^*\times \C^*$ or Fujita's $H[-1,\,0,\,-1]$. 
\end{proof}

\begin{theorem}\label{Thm: Kappa 1, pi_2 finite rank}
Let $X$ be a smooth complex affine surface with $\Kbar(X)\,=\,1$ such that $\pi_1(X)$ is infinite and
$\dim \pi_2(X)\otimes_{\Z}\Q\, <\, \infty$. Then $X$ admits a unique $\C^*$-fibration $f\,:\, X\, \longrightarrow\,B$ onto a
smooth algebraic curve $B$ not isomorphic to $\PP^1$ such that all the fibers of $f$, if taken with reduced structures, are isomorphic to $\C^*$. This assertion is equivalent to the fact that
$e(X)\,=\,0$. Consequently, $X$ is an Eilenberg-MacLane $K(\pi,\,1)$-space.
\end{theorem}

\begin{proof}
Since $\Kbar(X)\,=\,1$, by Kawamata's result, \cite{Ka}, \cite[Theorem (6.11)]{Fuj1}, there exists a $\C^\ast$-fibration $f\,:\, X\, 
\longrightarrow\,B$ onto a smooth algebraic curve $B$. It is well-known that this $\C^\ast$-fibration is also unique. Rest follows immediately from Theorem \ref{Surfaces admitting C^*-fibration with pi_2 of finite rank}. (Note that Remark \ref{Rem: e=0 implies all fibers C* with reduced structure} is also used here).
\end{proof} 

\subsection{Conclusions}
 
Theorems \ref{Thm: Kappa negative, pi_2 finite rank}, \ref{Thm: Kappa 0, pi_2 finite rank} and \ref{Thm: Kappa 1, pi_2 finite rank} 
together give the following:

\begin{prop}\label{Prop: Infinite pi_1 and finite rank pi_2 imply EM-ness for smooth affine non-general type surfaces}
Let $X$ be a smooth complex affine surface of the non-general type having an infinite fundamental group. If $\pi_2(X)$
is of finite rank, then $X$ is an Eilenberg--MacLane $K(\pi,\,1)$-space.
\end{prop}

Combining the results from previous sections we have the following main result of this paper.

\begin{theorem}\label{Classification of E-type affine surfaces of non-general type}
Let $X$ be a smooth complex affine surface of non-general type satisfying the finite homotopy rank-sum property. Then one of the following statements holds:
\begin{enumerate}
\item[(i)] $X$ is an Eilenberg--MacLane $K(\pi_1(X),\,1)$-space.

\item[(ii)] $\pi_1(X)\,=\,\Z/2\Z$, and $X$ satisfies one of the following, equivalent conditions:
\begin{enumerate}
    \item[(a)] the universal cover of $X$ is homotopically equivalent to $S^2$;
    \item[(b)] $e(X)\,=\,1$;
    \item[(c)] $X$ is a smooth affine $\Q$-homology plane.
\end{enumerate}

\item[(iii)] $X$ is homotopically equivalent to $S^2$.
\end{enumerate}
\end{theorem}

\begin{proof}
Suppose that $X$ has an infinite fundamental group. Then using Proposition \ref{Prop: Infinite pi_1 and finite rank pi_2 imply EM-ness 
for smooth affine non-general type surfaces} it follows that $X$ is an Eilenberg--MacLane $K(\pi,\,1)$-space.

Assume that $X$ has a finite fundamental group. If $X$ is simply connected, then by Theorem \ref{Characterization of Elliptic 
Homotopy Type Surfaces} either $X$ is contractible (this is also an Eilenberg--MacLane $K(\pi,\,1)$-space with $\pi\,=\,\pi_1(X)\,=\,(1)$) or 
$X$ is homotopic to $S^2$.

Let $\pi_1(X)$ be non-trivial finite. Then Corollary \ref{Cor:Finite Fundamental Group of Elliptic Type Affine Surface}
implies that $\pi_1(X)\,=\,\Z/2\Z$ and $e(X)\,=\,1$. Therefore $$b_2(X)\,=\,e(X)\,+\,b_1(X)\,-\,1\,=\,0,$$ and
hence $X$ is a smooth affine $\Q$-homology plane with
\begin{equation}\label{ec}
\pi_1(X)\,=\,\Z/2\Z.
\end{equation}
In this case, since the universal cover $\wt{X}$ of $X$ is a 
smooth affine surface of the elliptic homotopy type, it turns out that $\wt{X}$ cannot be contractible. Indeed, for
$\wt{X}$ being contractible $X$ becomes an 
Eilenberg--MacLane $K(\pi,\,1)$-space which must have a torsion-free fundamental group, but this contradicts
\eqref{ec}. Therefore, in this case, the universal cover of $X$ is homotopic to $S^2$. This completes the proof.
\end{proof}

\begin{example}\label{Ex: E-type affine surfaces of non-general type}

We will show that each one of the three classes of smooth affine surfaces described in Theorem 
\ref{Classification of E-type affine surfaces of non-general type} is actually non-empty.

Theorem \ref{Result : Classification} gives a complete classification of the smooth affine non-contractible Eilenberg--MacLane surfaces of non-general type.

An example of a surface described in (ii) of Theorem \ref{Classification of E-type affine surfaces of non-general type} is the 
hypersurface $H$ in $\A^3$ defined by
the equation $z^2\,=\,x(xy-1)$. Define a morphism 
$\varphi\,:\, H\, \longrightarrow\,\A^1$ by $(x,\,y,\,z)\, \longmapsto\, x$. Clearly, $\varphi$ is an $\A^1$-fibration with 
only one singular fiber over the origin in $\A^1$ such that scheme-theoretically $\varphi^*(0)\,=\,2\A^1$. Thus by Lemma \ref{Lem: Gang's Generalization}, it 
follows that $\pi_1(H)\,=\,\Z/2\Z$, and by Suzuki's formula (cf. \cite[Proposition 2]{Suz}) it turns out that $e(H)\,=\,1$. Therefore $H$ is uniformized by a
smooth affine simply connected surface, say $X$, which is a $2$-fold covering of $H$. Hence $e(X)\,=\,2\,\cdot\,e(H)\,=\,2$ and thus $X$ is a Moore $M(\Z,\,
2)$-space, i.e, $X$ is homotopic to $S^2$. This shows that $H$ satisfies the finite homotopy rank-sum property.

An example of a surface as described in (iii) of Theorem \ref{Classification of E-type affine surfaces of non-general type} is the 
hypersurface in $\A^3$ defined by $x^2+y^2+z^2\,=\,1$. We call this the \emph{complex $2$-sphere} in $\C^3$. It is 
well-known that this surface is homotopy equivalent to a real $2$-sphere in $\R^3$. In fact, the real $2$-sphere is a deformation 
retraction of the complex $2$-sphere.

Another example of a surface as described in (iii) of Theorem \ref{Classification of 
E-type affine surfaces of non-general type} is $V\,:=\,\PP^1\times \PP^1-\Delta$, where $\Delta$ is the diagonal of $\PP^1\times \PP^1$. 
Since $\Delta$ is an ample divisor in $\PP^1\times \PP^1$, the complement $V$ is a smooth affine surface. Considering the usual 
$\PP^1$-ruling over $\PP^1$ in the projective surface $\PP^1\times \PP^1$, we can see that $\Delta$ is a smooth rational curve which 
meets every vertical $\PP^1$ once transversally.
Thus $V$ is an $\A^1$-bundle over $\PP^1$, which implies that $V$ is homotopic to $\PP^1\,=\, S^2$.
\end{example}

\section{On homotopy groups of affine homology planes}\label{se7}

In this section, we will analyze the finite rank-sum property of higher homotopy groups for a very special class of smooth algebraic
surfaces, namely the $\Q$-homology planes. It is known that $\Q$-homology planes are affine (cf. \cite[2.5]{Fuj2}).

\begin{prop}\label{Prop: No non-contractible smooth EM Q-homology plane}
A smooth $\Q$-homology plane of non-general type is Eilenberg--MacLane if and only if it is contractible.
\end{prop}

\begin{proof}
Let $X$ be a smooth $\Q$-homology plane of non-general type which is an Eilenberg--MacLane space. If $X$ is non-contractible with 
$\Kbar(X)\,=\,-\infty$, then using the classification result of \cite{GGH} (cf. Theorem \ref{Result : Classification}(1)), $X$ is an 
$\A^1$-bundle over a smooth algebraic curve $B$, not isomorphic to $\A^1$ or $\PP^1$. Therefore $X$ is homotopy equivalent to $B$. 
Since $B$ is not isomorphic to $\A^1$ or $\PP^1$, we have $H_1(X;\,\Q)\,\cong \, H_1(B;\,\Q)\,\not=\, 0$. This 
contradicts the fact that $X$ is a $\Q$-homology plane. Hence $X$ is contractible, and hence $X$ is isomorphic to $\A^2$ in this case (see Remark \ref{Rem: Smooth Z-homology planes}). If
$\Kbar(X)\, \in\, \{0,\, 1\}$, and $X$ is non-contractible, then being Eilenberg--MacLane, it follows from the classification
result of \cite{GGH} (cf. Theorems \ref{Result : Classification}(2), \ref{Result : Classification}(3)) that $e(X)\,=\,0$. But
$X$ being a $\Q$-homology plane has Euler characteristic $1$. Hence $X$ is contractible whenever $\Kbar(X)\,=\,0$ or $1$. This
completes the proof.
\end{proof}

\begin{remark}\label{Rem: Smooth Z-homology planes}
By a result of Miyanishi (cf. \cite[Chapter 3, Theorem 4.7.1(3)]{Miy2}), a smooth $\Z$-homology plane with $\Kbar\,=\,-\infty$ is isomorphic to 
$\A^2$. In \cite{Fuj2}, Fujita listed all possible boundary divisors of a smooth NC-minimal affine surface with $\Kbar\, = \,0$; a 
striking consequence of his work is the result that there is no smooth $\Z$-homology plane (in particular, contractible)
with $\Kbar\,=\,0$ (see also \cite[Chapter 3, Theorem 4.7.1(1)]{Miy2}). So
Proposition \ref{Prop: No non-contractible smooth EM Q-homology plane} actually 
proves that if $X$ is a smooth $\Q$-homology plane of non-general type which is also an Eilenberg--MacLane space, then either $X$ is 
isomorphic to $\A^2$ or it is a contractible surface with $\Kbar\,=\,1$.
\end{remark}

We now generalize the result proved in Proposition \ref{Prop: No non-contractible smooth EM Q-homology plane} as follows.

\begin{cor}\label{Cor: Smooth Q-homology plane of E-type}
Let $X$ be a smooth $\Q$-homology plane of non-general type satisfying the finite homotopy rank-sum property. Then either $X$ is contractible or a smooth affine $\Q$-homology plane with $\pi_1(X)\,=\,\Z/2\Z$.
\end{cor}

\begin{proof}
This follows from Theorem \ref{Classification of E-type affine surfaces of non-general type} and Proposition \ref{Prop: No 
non-contractible smooth EM Q-homology plane}.
\end{proof}

\begin{cor}\label{Cor: No non-contractible smooth Z-homology plane of E-type}
Let $X$ be a smooth $\Z$-homology plane of non-general type satisfying the finite homotopy rank-sum property. Then $X$ is contractible.
\end{cor}

\begin{proof}
Since $X$ is a $\Q$-homology plane, Corollary \ref{Cor: Smooth Q-homology plane of E-type} says that $\pi_1(X)\,=\,\Z/2\Z$
if $X$ is non-contractible. But $\pi_1(X)\,\not=\,\Z/2\Z$ as $X$ is a $\Z$-homology plane.
\end{proof}

From the above, the following natural question arises.

\begin{question}
Is there any smooth affine $\Z$-homology plane which is an Eilenberg--MacLane space but not contractible?
\end{question}

The authors do not know any such example.

\begin{remark}
It is important to note that smooth contractible surfaces of $\Kbar\,=\,1$ do exist, see e.g. \cite[III.4.8.3, III.4.8.4]{Miy2}.
\end{remark}

\section*{Acknowledgements} 

The authors warmly thank the referee for a very careful reading of the manuscript and making many very useful suggestions to improve 
the exposition.
The authors would like to thank R. V. Gurjar for informing us about the content of the reference \cite{Ham}. The authors are also thankful to A. J. Parameswaran for 
his useful discussions at the beginning of this project.
The first-named author acknowledges the support of a J. C. Bose Fellowship (JBR/2023/000003).

\section*{Statements and Declarations}

There is no conflict of interests regarding this manuscript. No funding was received for it.
All authors contributed equally. No data were generated or used.


\begin{thebibliography}{ZZZZZ}

\bibitem[AB]{Amo-Bis} J. Amor{\'o}s and I. Biswas, Compact Kähler manifolds with elliptic homotopy type,
{\it Adv. Math.} {\bf 224} (2010), no. 3, 1167--1182.
 
\bibitem[AF]{And-Fra} A. Andreotti and T. Frankel, The Lefschetz theorem on hyperplane sections, {\it Ann. of Math.} {\bf 69} 
(1959), 713--717.

\bibitem[Ch]{Cha} T. C. Chau, A note concerning Fox's paper on Fenchel's conjecture, {\it Proc. Amer. Math. Soc.} {\bf 88} (1983), no. 
4, 584--586.

\bibitem[Dy]{Dy} M. Dyer, Rational homology and Whitehead products, {\it Pacific J. Math.} {\bf 40} (1972), 59--71.

\bibitem[FHT1]{Fel-Hal-Tho1} Y. F{\'e}lix, S. Halperin and J. -C. Thomas, The homotopy Lie algebra for finite complexes, {\it Inst. 
Hautes {\'E}tudes Sci. Publ. Math.} No. {\bf 56} (1982), 179--202.

\bibitem[FHT2]{Fel-Hal-Tho2} Y. F{\'e}lix, S. Halperin and J. -C. Thomas, \emph{Rational homotopy theory}, Graduate Texts in Mathematics, Volume 205, Springer-Verlag, New York, 2001.

\bibitem[FrH]{Fri-Hal} J. B. Friedlander and S. Halperin, An arithmetic characterization of the rational homotopy groups of 
certain spaces, {\it Invent. Math.} {\bf 53} (1979), no. 2, 117--133.
	
\bibitem[Fu1]{Fuj1} T. Fujita, On Zariski problem, {\it Proc. Japan Acad.} {\bf 55} (1979), 106--110.

\bibitem[Fu2]{Fuj2} T. Fujita, On the topology of noncomplete algebraic surfaces, {\it J. Fac. Sci.
Univ. Tokyo Sec. IA Math.} {\bf 29} (1982), no.3, 503--566.

\bibitem[Ga]{Gan} X. Gang, $\pi_1$ of elliptic and hyperelliptic surfaces, {\it Internat. J. Math.} {\bf 2} (1991), no. 5, 599--615.
	
\bibitem[Gi]{Gie} B. Giesecke, Simpliziale Zerlegung abz{\"a}hlbarer analytischer R{\"a}ume, {\it Math. Zeit.} {\bf 83} (1964), 177--213.

\bibitem[GGH]{GGH} R. V. Gurjar, S. R. Gurjar and B. Hajra, Eilenberg--MacLane Spaces in Algebraic Surface Theory, {\it Geom. Dedicata} {\bf 217} (2023), no. 31.

\bibitem[GMM]{GMM} R. V. Gurjar, K. Masuda and M. Miyanishi, \emph{Affine space fibrations}, De Gruyter Studies in Mathematics, {\bf 79}, 
De Gruyter, Berlin, 2021.

\bibitem[Ham]{Ham} H. A. Hamm, Zur Homotopietyp Steinscher R{\"a}ume, {\it J. Reine Angew. Math.} {\bf 338} (1983), 121--135.

\bibitem[Hat]{Hat} A. Hatcher, \emph{Algebraic topology}, Cambridge University Press, Cambridge, 2002.
 
\bibitem[Iit]{Iit} S. Iitaka, \emph{Algebraic Geometry}, Graduate Texts in Math., Springer-
Verlag 76, 1981.

\bibitem[KZ]{Kal-Zai} S. Kaliman and M. Zaidenberg, Affine modifications and affine hypersurfaces with a very transitive 
automorphism group, {\it Transform. Gr.} {\bf 4} (1999), no. 1, 53--95.

\bibitem[Ka]{Ka} Y. Kawamata, On the classification of non-complete algebraic surfaces, {\it Proc. Copenhagen Summer Meeting in 
Algebraic Geometry}, Lecture Notes in Mathematics, No. 732, 215--232, Berlin-Heiderberg-New York, Springer, 1978.

\bibitem[KK]{Kla-Kre} S. Klaus and M. Kreck, A quick proof of the rational Hurewicz theorem and a computation of the rational 
homotopy groups of spheres, {\it Math. Proc. Cambridge Philos. Soc.} {\bf 136} (2004), no. 3, 617--623.

\bibitem[Ko]{Koj} H. Kojima, Open rational surfaces with logarithmic Kodaira dimension zero, {\it Internat. J. Math.} {\bf 10} (1999), 
no. 5, 619--642.

\bibitem[Mi1]{Miy1} M. Miyanishi, \emph{Noncomplete algebraic surfaces}, Lecture Notes in Mathematics, {\bf 857}, Springer-Verlag, 
Berlin-New York, 1981.
	
\bibitem[Mi2]{Miy2} M. Miyanishi, \emph{Open Algebraic Surfaces}, CRM Monograph Series, {\bf 12}, American Mathematical Society, 
Providence, RI, 2001.
	
\bibitem[MS1]{Miy-Sug} M. Miyanishi and T. Sugie, Affine surfaces containing cylinderlike open sets, {\it J. Math. Kyoto Univ.} {\bf 
20} (1980), 11--42.

\bibitem[MS2]{Miy-Sug1} M. Miyanishi and T. Sugie, Homology planes with quotient singularities, {\it J. Math. Kyoto Univ.} {\bf 31} 
(1991), no. 3, 755--788.

\bibitem[MT]{Mos-Tan} R. E. Mosher and M. C. Tangora, \emph{Cohomology operations and applications in homotopy theory}, Harper \& Row, Publishers, New York-London, 1968.

\bibitem[Na]{Nar} R. Narasimhan, On the homology groups of Stein spaces, {\it Invent. Math.} 2 (1967), 377--385. 
	
\bibitem[No]{Nor} M. V. Nori, Zariski conjecture and related problems, {\it Ann. Sci. {\'E}cole Norm. Sup.} {\bf 16} (1983), 305--344.

\bibitem[Ser]{Ser} J.-P. Serre, Groupes d'homotopie et classes de groupes Ab{\'e}liens, {\it Ann. of Math. (2)} {\bf 58} (1953), 258--294.

\bibitem[Sh1]{Sha} I. R. Shafarevich, \emph{Basic Algebraic Geometry}, Grundlehren der mathematischen Wissenschaften, Vol. {\bf 213}, 1974.

\bibitem[Sh2]{Sha1} I. R. Shafarevich, \emph{Basic Algebraic Geometry 1 - Varieties in projective space}, Second edition. Translated 
from the 1988 Russian edition and with notes by Miles Reid, Springer-Verlag, Berlin, 1994.

\bibitem[Spa]{Spa} E. H. Spanier, \emph{Algebraic Topology}, M$_c$Graw-Hill Series in Higher Mathematics, New York-Toronto, Ontario-London, 1966.

\bibitem[Sug]{Sug} T. Sugie, On a characterization of surfaces containing cylinderlike open sets, {\it Osaka J. Math.} {\bf 17} (1980), 363--376.
	
\bibitem[Suz]{Suz} M. Suzuki, Sur les op{\'e}rations holomorphes du groupe additif complexe sur l'espace de deux variables 
complexes, {\it Ann. Sci. {\'E}cole Norm. Sup.} {\bf 10} (1977), no. 4, 517--546.

\bibitem[Wh1]{Whi1} J. H. C. Whitehead, Combinatorial homotopy I, {\it Bull. Amer. Math. Soc.} {\bf 55} (1949), 213--245.

\bibitem[Wh2]{Whi2} J. H. C. Whitehead, Combinatorial homotopy II, {\it Bull. Amer. Math. Soc.} {\bf 55} (1949), 453--496.

\end{thebibliography}
\end{document}